\documentclass[options]{amsart}

\usepackage{latexsym,ifthen,amssymb}
\usepackage{latexsym,ifthen,amssymb}
\usepackage[numbers]{natbib}
\usepackage[toc,page,title,titletoc,header]{appendix}
\usepackage{multirow}
 \usepackage{longtable}
 \usepackage{rotating}
 \usepackage{graphicx, amsmath, amsthm, amssymb,float,setspace,color, multirow}
\usepackage{enumitem}
\usepackage{graphicx}
\usepackage{bm}
\usepackage{subfigure}
\usepackage{float}
\usepackage{rotating}
\usepackage{latexsym,ifthen,amssymb}
\usepackage{natbib}
\usepackage[toc,page,title,titletoc,header]{appendix}
\usepackage{multirow}
 \usepackage{longtable}
 \usepackage{rotating}
 \usepackage{graphicx, amsmath, amsthm, amssymb,float,setspace,color, multirow}
\usepackage{tikz}
\usepackage{hyperref}
\usepackage{enumitem}
\usepackage[margin=1cm]{caption}

\hypersetup{
	pdfborder={0 0 0}
}
\usetikzlibrary{decorations.markings}

\usepackage{url}
\usepackage[margin=1in]{geometry}
\usepackage{xcolor}	
\usepackage{soul}

\newcommand{\Nb}{\mathbb{N}}

\newcommand{\Psf}{\mathsf{P}}
\newcommand{\Qsf}{\mathsf{Q}}

\newcommand{\Gcal}{\mathcal{G}}

\newcommand{\Vcal}{\mathcal{V}}

\newcommand{\Lcal}{\mathcal{L}}

\newcommand{\Pcal}{\mathcal{P}}

\newcommand{\Rcal}{\mathcal{R}}
\newcommand{\Scal}{\mathcal{S}}

\newcommand{\Tcal}{\mathcal{T}}

\newcommand{\uh}{{\upharpoonright}}

\renewcommand{\setminus}{\smallsetminus}

\def\qt#1{``#1''}%

\newcommand{\s}[1]{\ensuremath{\sf{#1}}}

\DeclareMathOperator{\rca}{\s{RCA}_0}

\DeclareMathOperator{\aca}{\s{ACA}}
\DeclareMathOperator{\wkl}{\s{WKL}}

\DeclareMathOperator{\rt}{\s{RT}}

\DeclareMathOperator{\crt}{\s{CRT}}
\DeclareMathOperator{\rrt}{\s{RRT}}

\DeclareMathOperator{\ads}{\s{ADS}}

\DeclareMathOperator{\cac}{\s{CAC}}

\DeclareMathOperator{\fs}{\s{FS}}

\DeclareMathOperator{\sher}{\s{SHER}}
\DeclareMathOperator{\emo}{\s{EM}}

\DeclareMathOperator{\dom}{dom}
\newtheorem{theorem}{Theorem}[section]
\newtheorem{lemma}[theorem]{Lemma}

\newtheorem{corollary}[theorem]{Corollary}

\theoremstyle{definition}
\newtheorem{definition}[theorem]{Definition}
\newtheorem{statement}[theorem]{Statement}

\theoremstyle{remark}

\newtheorem {question}[theorem]{Question}
\numberwithin{equation}{section}

\newtheoremstyle{noparens}%
  {}{}%
{}{}%
{\bfseries}{.}%
{ }%
{\thmname{#1}\thmnumber{ #2}\thmnote{ #3}}
\theoremstyle{noparens}
\newtheorem*{question*}{Question}
\newtheorem*{theorem*}{Theorem}

\newcommand{\largeo}{\s{SCA}\mbox{-}\s{RT}}
\newcommand{\largep}[1]{#1\mbox{-}\s{SCA}\mbox{-}\s{RT}}
\newcommand{\lcelargeo}{\s{CA}\mbox{-}\s{RT}}
\newcommand{\lcelargep}[1]{#1\mbox{-}\s{CA}\mbox{-}\s{RT}}
\newcommand{\explicitlarge}{\s{LARGE}}
\newcommand{\lceexplicitlarge}{\s{PACKED}}
\newcommand{\arithscart}{\s{ARITH}\mbox{-}\s{SCA}\mbox{-}\s{RT}}
\newcommand{\arithscartp}[1]{#1\mbox{-}\s{ARITH}\mbox{-}\s{SCA}\mbox{-}\s{RT}}
\newcommand{\ftt}{\mathtt{f}}

\title{Ramsey-like theorems and moduli of computation}

\author{Ludovic Patey}
\address{Institut Camille Jordan\\
Universit\'e Claude Bernard Lyon 1\\
43 boulevard du 11 novembre 1918\\
F-69622 Villeurbanne Cedex}
\email{ludovic.patey@computability.fr}

\thanks{}

\begin{document}

\begin{abstract}
Ramsey's theorem asserts that every $k$-coloring of $[\omega]^n$ admits an infinite monochromatic set. Whenever $n \geq 3$, there exists a computable $k$-coloring of $[\omega]^n$ whose solutions compute the halting set. On the other hand, for every computable $k$-coloring of $[\omega]^2$ and every non-computable set $C$, there is an infinite monochromatic set $H$ such that $C \not \leq_T H$. The latter property is known as \emph{cone avoidance}.

In this article, we design a natural class of Ramsey-like theorems encompassing many statements studied in reverse mathematics. We prove that this class admits a maximal statement satisfying cone avoidance and use it as a criterion to re-obtain many existing proofs of cone avoidance.
This maximal statement asserts the existence, for every $k$-coloring of $[\omega]^n$, of an infinite subdomain $H \subseteq \omega$ over which the coloring depends only on the sparsity of its elements. This confirms the intuition that Ramsey-like theorems compute Turing degrees only through the sparsity of its solutions.
\end{abstract}

\maketitle

Ramsey's theorem asserts that every $k$-coloring of $[\omega]^n$ admits an infinite monochromatic set, where $[X]^n$ denotes the set of the unordered $n$-tuples over $X$. This theorem plays an central role in reverse mathematics, as Ramsey's theorem for pairs is historically the first example of a theorem escaping the structural phenomenon known as the \qt{Big Five} phenomenon. See Simpson~\cite{Simpson2009Subsystems} for a reference on the early reverse mathematics. In his celebrated theorem, Seetapun~\cite{Seetapun1995strength} proved that Ramsey's theorem for pairs admits \emph{cone avoidance}, that is, for any fixed non-computable set $C$, every computable $k$-coloring of $[\omega]^2$ admits an infinite monochromatic set $H$ which does not compute $C$. Since then, many consequences of Ramsey's theorem have been studied from a computability-theoretic viewpoint, including the Erd\H{o}s-Moser theorem~\cite{Bovykin2005strength}, the Ascending Descending sequence principle~\cite{Hirschfeldt2007Combinatorial}, the free set and thin set theorems~\cite{Cholak2001Free} and the rainbow Ramsey theorem~\cite{Csima2004Bounding}. Seeing these statements as problems, in terms of \emph{instances} and \emph{solutions}, the community studied basis theorems for various computability-theoretic properties, including cone avoidance for computable instances, but also for arbitrary instances. This latter property is known as \emph{strong cone avoidance}.

In this article, we generalize the above analysis by designing a general class of Ramsey-like statements encompassing the above examples, and providing general criteria to decide whether any such statement admits (strong) cone avoidance. We start with a short survey on cone avoidance for Ramsey's theorem in Section~\ref{sect:survey-encodability}. Then, we define in Section~\ref{sect:ramsey-like} the class of \emph{Ramsey-like} statements to be those of the form \qt{For every coloring $f : [\omega]^n \to k$, there is an infinite set $H \subseteq \omega$ avoiding some set of forbidden patterns relative to $f$.} This class contains Ramsey's theorem, but also the Erd\H{o}s-Moser theorem. We prove that this class contains a maximal statement admitting strong cone avoidance ($\largeo^n_k)$ and cone avoidance ($\lcelargeo^n_k$) and characterize the statements admitting strong cone avoidance and cone avoidance as those identically reducible to $\largeo^n_k$ and $\lcelargeo^n_k$, respectively. In Section~\ref{sect:promise-ramsey-like}, we define the class of \emph{promise Ramsey-like} statements which generalizes the class of Ramsey-like statements by restricting the instances to those satisfying some property. These statements as then of the form \qt{For every coloring $f : [\omega]^n \to k$ such that $\omega$ avoids some set of forbidden patterns relative to $f$, there is an infinite set $H \subseteq \omega$ avoiding some other set of forbidden patterns relative to $f$.} This enables us to express statements about other structures, such as linear orders and partial orders, including the Ascending Descending sequence and the Chain Antichain principle, but also statements about $\omega$-colorings over $[\omega]^n$, such as the free set or the rainbow Ramsey theorem. In Section~\ref{sect:applications}, we apply the previous analysis to reprove many existing theorems, including cone avoidance of Ramsey's theorem for pairs~\cite{Seetapun1995strength}, strong cone avoidance of Ramsey's theorem for singletons~\cite{Dzhafarov2009Ramseys}, strong cone avoidance of the Erd\H{o}s-Moser theorem~\cite{PateyCombinatorial}, strong cone avoidance of the thin set and free set theorems~\cite{Wang2014Some}, among others. Last, in Section~\ref{sect:open-questions}, we state some remaining open questions and suggest further developments.

Put aside the practical application of this general framework to obtain forcing-free proofs of cone and strong cone avoidance, the actual statements of $\largeo^n_k$ and $\lcelargeo^n_k$ and the resulting decidability criteria provide some further insights on the nature of computation of Ramsey-like statements. One can see a function $\mu : \omega \to \omega$ as a measure of largeness of the intervals over $\omega$, by saying that $[x, y]$ is $\mu$-large if $\mu(x) \leq y$. The statements $\largeo^n_k$ and $\lcelargeo^n_k$ are both of the form \qt{For every coloring $f : [\omega]^n \to k$, there is a function $\mu : \omega \to \omega$ and an infinite set $H \subseteq \omega$ such that for every $D \in [H]^n$, $f(D)$ depends only on the $\mu$-largeness analysis over $D$.} This gives further evidence that Ramsey-like theorems get their computational power out of the sparsity of the solutions, and more generally that Ramsey-like theorems compute through moduli.

\section{A short survey on encodability by Ramsey's theorem}\label{sect:survey-encodability}

For the sake of clarity, we will adopt a thematic presentation, independently of the historical aspects of these discoveries.
Ramsey's theorem is a combinatorial theorem at the foundation of \emph{Ramsey's theory}. This theory studies the conditions under which given a sufficiently large amount of data, one can see the emergence of some structure.
A $k$-coloring of $[\omega]^n$ is a function $f : [\omega]^n \to k$. A set $H \subseteq \omega$ is \emph{$f$-homogeneous} if $f$ is constant on $[H]^n$.

\begin{statement}[Ramsey's theorem]
$\rt^n_k$: Every coloring $f : [\omega]^n \to k$ admits an infinite $f$-homogeneous set.
\end{statement}

Ramsey's theorem plays a central role in reverse mathematics. It is historically the first theorem which does not belong to the empirical structural observation of mathematics.
This motivated the computability-theoretic analysis of Ramsey's theorem and its consequences.

One can see Ramsey's theorem as a mathematical \emph{problem}, expressed in terms of \emph{instances} and \emph{solutions}. Here, an instance of $\rt^n_k$ is a coloring $f : [\omega]^n \to k$. A solution to an $\rt^n_k$-instance $f$ is an infinite $f$-homogeneous set. The computable analysis of Ramsey's theorem consists of, given an instance of $\rt^n_k$, studying the complexity of its solutions from a computable and a proof-theoretic viewpoint. This study started with Jockusch~\cite{Jockusch1972Ramseys}, who proved that every computable instance of $\rt^n_k$ admits an arithmetical solution.

In this article, we are interested in the ability of Ramsey's theorem to compute Turing degrees, that is, the existence of (computable or not) instances of $\rt^n_k$ such that every solution computes a fixed Turing degree.

\subsection{Encodability by Ramsey's theorem}

Given a problem $\Psf$ with instances and solutions, we say that a set $A$ is \emph{encodable by $\Psf$}, or $\Psf$-encodable, if there is an (arbitrary) instance of $\Psf$ such that every solution computes~$A$.

The study of $\rt^n_k$-encodable sets is closely bound to the ability of Ramsey's theorem to compute fast-growing functions.
Suppose a set $A$ admits a \emph{modulus}, that is, a function $\mu : \omega \to \omega$ such that every function dominating $\mu$ computes~$A$. Then, define the $\rt^2_2$-instance $f : [\omega]^2 \to 2$ for each $x < y$ by $f(x, y) = 1$ if and only if $y \geq \mu(x)$, that is, if the interval $[x, y]$ is sufficiently large. Every infinite $f$-homogeneous set $H$ must be of color~1, and its \emph{principal function} $p_H : \omega \to \omega$, which on $n$ associates the $n$th element of~$H$, dominates~$\mu$ and therefore computes~$A$. 

Such an argument shows that if a set $A$ admits a modulus, then it is $\rt^n_k$-encodable, for $n \geq 2$ and $k \geq 2$. Moreover, this encodability is witnessed by constructing an instance of $\rt^n_k$ whose solutions are sparse enough so that their principal functions are fast-growing.
The sets admitting a modulus have been studied by Groszek and Slaman~\cite{Groszek2007Moduli}, who proved that these are precisely the hyperarithmetical sets.

On the other hand, Solovay~\cite{Solovay1978Hyperarithmetically} proved that no other degree can be encoded by Ramsey's theorem, using the notion of computable encodability. Given a set $X \subseteq \omega$, we let $[X]^\omega$ be the collection of all the infinite subsets of~$X$.

\begin{definition}
A set $A$ is \emph{computably encodable} if for every set $X \in [\omega]^\omega$, there is a set $Y \in [X]^\omega$ such that $Y \geq_T A$.
\end{definition}

Suppose that a set $A$ is computed by an instance $f : [\omega]^n \to k$ of $\rt^n_k$, in other words, every infinite $f$-homogeneous set computes~$A$.
In particular, since for every set $X \in [\omega]^\omega$, there is an infinite $f$-homogeneous set $H \subseteq X$, then the set $A$ is computably encodable. The following equivalence proves that the only Turing degrees which can be computed by an instance of Ramsey's theorem are the ones which admit a modulus, hence those who can be computed using fast-growing functions.

\begin{theorem}[Solovay~\cite{Solovay1978Hyperarithmetically}, Groszek and Slaman~\cite{Groszek2007Moduli}]
Given a set $A$, the following are equivalent
\begin{itemize}
	\item[(a)] $A$ is computably encodable
	\item[(b)] $A$ is hyperarithmetic
	\item[(c)] $A$ admits a modulus
\end{itemize}
\end{theorem}

One case however remains, namely, Ramsey's theorem for singletons. Intuitively, one cannot build an instance of $\rt^1_2$ whose solutions are sparse everywhere, so that their principal functions are sufficiently fast-growing. Indeed, given a $\rt^1_2$-instance $f : \omega \to 2$, in order to ensure that the $f$-homogeneous sets for color~0 are all sparse, one has to set $f(x) = 1$ for many $x \in \omega$. But then, one can construct non-sparse $f$-homogeneous sets for color~1. Actually, Ramsey's theorem for singletons has no encodability power, as formalized by the notion of \emph{strong cone avoidance}.

\begin{definition}
A problem $\Psf$ admits \emph{strong cone avoidance} if for every set $Z$, every $C \not \leq_T Z$ and every $\Psf$-instance $X$, there is a $\Psf$-solution $Y$ to $X$ such that $C \not \leq_T Z \oplus Y$.
\end{definition}

Building on the work of Seetapun and Slaman~\cite{Seetapun1995strength} and of Cholak, Jockusch and Slaman~\cite{Cholak2001strength}, Dzhafarov and Jocksuch~\cite{Dzhafarov2009Ramseys} proved that Ramsey's theorem for singletons ($\rt^1_k$) admits strong cone avoidance, thereby completing the picture of which sets are encodable by Ramsey's theorem. To summarize, a set is encodable by $\rt^n_k$ with $n \geq 2$ if and only if it is hyperarithmetic, and is encodable by $\rt^1_k$ if and only if it is computable.

\subsection{Encodability by computable instances}\label{sect:enco-by-computable-instances} One may naturally want to refine the previous analysis, and study the $\rt^n_k$-encodable sets with respect to the computational complexity of its instances. Jockusch~\cite{Jockusch1972Ramseys} proved that every computable instance of $\rt^n_k$ admits an arithmetical solution. On the lower bound side, he proved that for every $n \geq 3$, there is computable instance of $\rt^n_2$ such that every solution computes $\emptyset^{(n-2)}$. The proof is a simple effectivization of the previous section.

	On the other direction, Seetapun and Slaman~\cite{Seetapun1995strength} proved that Ramsey's theorem for pairs ($\rt^2_k$) has no encodability power when restricted to computable instances, in the following sense.
	
\begin{definition}
A problem $\Psf$ admits \emph{cone avoidance} if for every set $Z$, every $C \not \leq_T Z$ and every \emph{$Z$-computable} $\Psf$-instance $X$, there is a $\Psf$-solution $Y$ to $X$ such that $C \not \leq_T Z \oplus Y$.
\end{definition}
	
While strong cone avoidance expresses the \emph{combinatorial} failure of $\Psf$ to encode any non-computable set, cone avoidance only expresses the \emph{computational} weakness of $\Psf$. There is a deep link between the combinatorial features of $\rt^n_k$ and the computational features of $\rt^{n+1}_k$, as expressed by Cholak and Patey~\cite[Theorem 1.5]{Cholak2019Thin}. One can deduce cone avoidance of $\rt^2_k$ from strong cone avoidance of $\rt^1_k$, although historically, Seetapun and Slaman~\cite{Cholak2001strength} first proved cone avoidance of~$\rt^2_k$.
Cholak, Jockusch and Slaman~\cite[Theorem~12.2]{Cholak2001strength} (proof fixed in~\cite[Appendix~A]{Hirschfeldt2016notions}), relativized cone avoidance of $\rt^2_k$ to prove that for every $n \geq 2$, if a set $A$ is not $\Delta^0_{n-1}$, then every computable instance of $\rt^n_k$ admits a solution $H$ such that $A \not \leq_T H$. This completes the study of the sets encodable by computable instances of $\rt^n_k$. Indeed, for every $n \geq 2$, a set $A$ is encodable by a computable instance of $\rt^n_k$ if and only if $A$ is $\Delta^0_{n-1}$.

\subsection{Encodability by the thin set theorems}
The previous sections give a complete picture of which sets are encodable by Ramsey's theorem, with or without restricting the complexity of the instances. This could be the end of the story. However, Wang~\cite{Wang2014Some} surprisingly showed that by weakening the notion of homogeneity to allow sufficiently many colors in the solutions, one obtains strong cone avoidance.

\begin{statement}[Thin set theorem]
$\rt^n_{<\infty,\ell}$: For every coloring $f : [\omega]^n \to k$, there is an infinite set $H \subseteq \omega$ such that $|f[H]^n| \leq \ell$.
\end{statement}

Wang~\cite{Wang2014Some} proved that for every $n \geq 1$ and every sufficiently large $\ell \in \omega$, $\rt^n_{<\infty, \ell}$ admits strong cone avoidance. On the other hand, Cholak and Patey~\cite[Theorem 3.2]{Cholak2019Thin}, adapting Dorais et al.~\cite[Proposition 5.5]{Dorais2016uniform}, proved that if $\ell < 2^{n-1}$, $\rt^n_{<\infty, \ell}$ can still compute arbitrarily fast-growing functions, and therefore computes all the hyperarithmetic sets. More precisely.

\begin{theorem}[Cholak and Patey]
Fix $n \geq 1$.
\begin{itemize}
	\item[(a)] For every function $\mu : \omega \to \omega$, there is an instance of $\rt^n_{<\infty, 2^{n-1}-1}$ such that every solution computes a function dominating~$\mu$.
	\item[(b)] If $A$ is not arithmetical, then every instance of $\rt^n_{<\infty, 2^{n-1}}$ has a solution which does not compute~$A$.
\end{itemize}
\end{theorem}

Using the previous sections, whenever a set $A$ is not hyperarithmetical, it is not $\rt^n_2$-encodable, hence not $\rt^n_{<\infty, \ell}$-encodable for any $\ell \geq 1$. Whenever $A$ is hyperarithmetical, but not arithmetical, then it is $\rt^n_{<\infty, \ell}$-encodable if and only if $\ell < 2^{n-1}$. The case of the the arithmetical sets must be treated independently.

Consider the halting set $\emptyset'$. The standard modulus of $\emptyset'$ is defined by letting $\mu(n)$ be the smallest time $t$ at which for every $e < n$, if $\Phi_e(e)\downarrow$, then $\Phi_e(e)$ halts before stage~$t$. One can in particular computably approximate the function $\mu$ from below. This is the notion of left-c.e.\ function.

\begin{definition}
A function $\mu : \omega \to \omega$ is \emph{left-c.e.} if there is a uniformly computable sequence of functions $\mu_0, \mu_1, \dots$ with $\mu_s : \omega \to \omega$ such that for every $s \in \omega$, $\mu_s \leq \mu_{s+1}$, and for every $x \in \omega$, $\lim_s \mu_s(x) = \mu(x)$.
\end{definition}

The notion relativizes, and we say that a function is \emph{left-$X$-c.e.} if the sequence of functions is uniformly $X$-computable. When mentioning a left-c.e.\ function, we will always assume that the sequence of its approximations is specified. This is why we will sometime talk about \emph{relative left-c.e.\ function} simply to say that a sequence of lower approximations of the function is fixed, no matter the effectiveness of the sequence.

Cholak and Patey~\cite{Cholak2019Thin} studied the threshold of $\ell$ under which, given a left-c.e.\ function $\mu$, there is an instance of $\rt^n_{<\infty, \ell}$ such that every solution computes a function dominating $\mu$. This happens to be exactly the \emph{Catalan sequence}, inductively defined by $C_0 = 1$ and 
$$
C_{n+1} = \sum_{i=0}^n C_i c_{n-i}
$$
In particular, $C_0 = 1$, $C_1 = 1$, $C_2 = 2$, $C_3 = 5$, $C_4 = 14$, $C_5 = 42$, $C_6 = 132$, $C_7 = 429$, $\dots$ Note that this sequence corresponds to the OEIS sequence A000108.

\begin{theorem}[Cholak and Patey~\cite{Cholak2019Thin}]
Fix $n \geq 1$.
\begin{itemize}
	\item[(a)] For every \emph{left-c.e.} function $\mu : \omega \to \omega$, there is an instance of $\rt^n_{<\infty, C_n-1}$ such that every solution computes a function dominating~$\mu$.
	\item[(b)] If $A$ is not computable, then every instance of $\rt^n_{<\infty, C_n}$ has a solution which does not compute~$A$.
\end{itemize}
\end{theorem}

By iterating the notion of left-c.e.\ function, one can $\rt^n_{<\infty, C_n-1}$-encode all the arithmetical sets for every $n \geq 2$.
This completes the picture of the $\rt^n_{<\infty, \ell}$-encodable sets, depending on the value of $n$ and $\ell$. There are actually three classes of $\rt^n_{<\infty, \ell}$-encodable sets: the computable, arithmetical, and hyperarithmetical sets. One can also deduce which sets are encodable by computable instances of $\rt^n_{<\infty, \ell}$, using the bridge between the combinatorics of $\rt^n_{<\infty, \ell}$ and the computations of $\rt^{n+1}_{<\infty, \ell}$. See Cholak and Patey~\cite{Cholak2019Thin} for this analysis.
 


\section{Ramsey-like theorems}\label{sect:ramsey-like}

As explained, $\rt^n_2$ does not admit strong cone avoidance for every $n \geq 2$. However, there exist some weakenings of Ramsey's theorem which admit strong cone avoidance. The thin set theorem is an example, but also the \emph{Erdos-Moser theorem}. Given a coloring $f : [\omega]^2 \to 2$, the Erdos-Moser theorem ($\emo$) asserts the existence of an infinite set $H \subseteq \omega$ over which $f$ is \emph{transitive}, that is, for every $x < y < z$ and $i < 2$, if $f(x, y) = i$ and $f(y, z) = i$ then $f(x, z) = i$. The author~\cite{PateyCombinatorial} proved that $\emo$ admits strong cone avoidance. Ramsey's theorem, the thin set theorem and the Erd\H{o}s-Moser theorem are all of the form \qt{For every coloring $f : [\omega]^n \to k$, there exists an infinite set $H \subseteq \omega$ avoiding some set of forbidden patterns relative to $f$.} In this section, we design a general class of Ramsey-like theorems, and provide a criterion to decide which statements admit strong cone avoidance.

\begin{definition}
Fix a countable collection of variables $x_0, x_1, \dots$
An \emph{$\rt^n_k$-pattern} $P$ is a finite conjunction of formulas of the form $\ftt(\{x_i : i \in D\}) = v$ for some $D \in [\omega]^n$ and $v < k$, where $\ftt$ is a function symbol of type $[\omega]^n \to k$. 
Given an actual coloring $f : [\omega]^n \to k$, a set of integers $E = \{n_0 < n_1 < \dots < n_{r-1}\}$ \emph{$f$-satisfies} an $\rt^n_k$-pattern $P \equiv \ftt(\{x_i : i \in D_0\}) = v_0 \wedge \dots \wedge \ftt(\{x_i : i \in D_{\ell-1}\}) = v_{\ell-1}$ if for every $s < \ell$, letting $E_s = \{n_i : i \in D_s\}$, $f(E_s) = v_s$.
A set $H \subseteq \omega$ \emph{$f$-meets} $P$  if $H$ contains an finite subset $f$-satisfying $P$. Otherwise, $H$ \emph{$f$-avoids} $P$.
\end{definition}

For example, $\ftt(x_5, x_6) = 0 \wedge \ftt(x_6, x_7) = 1 \wedge \ftt(x_5, x_7) = 0$ is an $\rt^2_2$-pattern.  

\begin{definition}
Given a collection $W$ of $\rt^n_k$-patterns, the \emph{$\rt^n_k$-like problem} $\rt^n_k(W)$ is the problem whose instances are colorings $f : [\omega]^n \to k$. An $\rt^n_k(W)$-solution to an instance $f$ is an infinite set $H \subseteq \omega$ $f$-avoiding every pattern in $W$.
\end{definition}

In particular, $\rt^2_2$ is the $\rt^2_2$-like problem $\rt^2_2(W_{\rt^2_2})$ with $W_{\rt^2_2} = \{ \ftt(x_0, x_1) = 0 \wedge \ftt(x_2, x_3) = 1, \ftt(x_0, x_1) = 1 \wedge \ftt(x_2, x_3) = 0, \ftt(x_0, x_2) = 0 \wedge \ftt(x_1, x_3) = 1, \ftt(x_0, x_2) = 1 \wedge \ftt(x_1, x_3) = 0 \}$. Similarly, $\emo$ is the $\rt^2_2$-like problem $\rt^2_2(W_{\emo})$ with $W_{\emo} = \{ \ftt(x_0, x_1) = 0 \wedge \ftt(x_1, x_2) = 0 \wedge \ftt(x_0, x_2) = 1,
		\ftt(x_0, x_1) = 1 \wedge \ftt(x_1, x_2) = 1 \wedge \ftt(x_0, x_2) = 0 \}$.

\subsection{True Ramsey-like problems}
We say that a problem is \emph{true} if every instance has a solution. Intuitively, Ramsey's theorem is the strongest structural property we can get out of a $k$-coloring of $[X]^n$, where $X$ is an abstract set. We formalize this intuition by proving that Ramsey's theorem is the maximal true Ramsey-like theorem.

\begin{definition}
Let $\Psf$ and $\Qsf$ be two problems with $\dom(\Psf) \subseteq \dom(\Qsf)$.
We say that $\Psf$ is \emph{identically reducible} to $\Qsf$ (written $\Psf \leq_{id} \Qsf$) 
if for every $I \in \dom(\Psf)$, every $\Qsf$-solution to $I$ is a $\Psf$-solution to $I$.
\end{definition}

\begin{theorem}\label{thm:rtnk-is-universal}
Fix $n, k \geq 1$ and a collection $W$ of $\rt^n_k$-patterns.
Then $\rt^n_k(W)$ is true if and only if $\rt^n_k(W) \leq_{id} \rt^n_k$.
\end{theorem}
\begin{proof}
$\Rightarrow$: Suppose $\rt^n_k(W) \not\leq_{id} \rt^n_k$. Let $f : [\omega]^n \to k$ be an instance of $\rt^n_k(W)$ and let $H$ be an infinite $f$-homogeneous set for some color $i < 2$ which $f$-meets some $P \in W$. Then $P$ is of the form $\ftt(\{x_j : j \in D_0\}) = i \wedge \dots \wedge \ftt(\{x_j : j \in D_{\ell-1}\}) = i$. Then, letting $g : [\omega]^n \to k$ be defined for every $D \in [\omega]^n$ by $g(D) = i$, $g$ has no $\rt^n_k(W)$-solution, therefore $\rt^n_k(W)$ is not true.

$\Leftarrow$: Suppose $\rt^n_k(W) \leq_{id} \rt^n_k$. Given an instance $f : [\omega]^n \to k$ of $\rt^n_k(W)$, by the classical Ramsey theorem, there is an infinite $f$-homogeneous set $H$. In particular, $H$ is an $\rt^n_k(W)$-solution to~$f$, so $f$ admits an $\rt^n_k(W)$-solution and $\rt^n_k(W)$ is true.
\end{proof}

Before starting the analysis of cone avoidance for Ramsey-like theorems, let us introduce an important concept which will be implicitly used all over the article.

\begin{definition} A problem $\Psf$ is \emph{zoomable} if for every $\Psf$-instance $I$ and every infinite set $X = \{x_0 < x_1 < \dots \}$, there is a $\Psf$-instance $I_X$ such that for every $\Psf$-solution $Y$ to $I_X$, $\{x_n : n \in Y \}$ is a $\Psf$-solution to $I$.
\end{definition}

One can see zoomable problems as saying that given an instance $I$ and an infinite set $X$, we can zoom in on the set $X$ and consider it as the new set $\omega$ by renaming the elements of $X$. Then, after having built a solution within $X$ seen as $\omega$, we can zoom out and see it as a subset of $X$. For example, Ramsey's theorem is a zoomable problem, while Hindman's theorem is not, since the zoom operation changes the semantics of the addition.

The following lemma which is implicitly used everywhere asserts that whenever a zoomable problem $\Psf$ admits strong cone avoidance, then one can apply this strong cone avoidance within the scope of a cone avoiding reservoir.

\begin{lemma}
Let $\Psf$ be a zoomable problem which admits strong cone avoidance.
For every set $Z$, every set $C \not \leq_T Z$, every infinite $Z$-computable set $X$ and every $\Psf$-instance $I$, there is a $\Psf$-solution $Y \subseteq X$ such that $C \not \leq_T Z \oplus Y$. 
\end{lemma}
\begin{proof}
Fix $Z$, $C$ and $X$. Since $\Psf$ is zoomable, there is a $\Psf$-instance $I_X$ such that for every $\Psf$-solution $Y$ to $I_X$, $\{x_n : n \in Y \}$ is a $\Psf$-solution to $I$. By strong cone avoidance of $\Psf$ applied to $I_X$, there is a $\Psf$-solution $Y$ to $I_X$ such that $C \not \leq_T Z \oplus Y$. The set $Y_X = \{x_n : n \in Y\}$ is a $Z \oplus Y$-computable $\Psf$-solution to $I$. In particular, $Y_X \subseteq X$ and $C \not \leq_T Z \oplus Y_X$.
\end{proof}

Similar lemmas can be proven for cone avoidance and strong cone avoidance for non-arithmetical cones. We will use these lemmas without any further mention.

\subsection{Strongly avoiding non-arithmetical cones} Recall that given a modulus $\mu : \omega \to \omega$ of a set $A$, one can define a coloring $f : [\omega]^2 \to 2$ by $f(x, y) = 1$ if and only if $\mu(x) \leq y$. Any infinite $f$-homogeneous set computes a function dominating~$\mu$. Cholak and Patey~\cite{Cholak2019Thin} generalized this coding in the following sense. 

\begin{definition}
A graph $\Gcal = (\{0, \dots, n-1\}, E)$ of size $n$ is a \emph{vector graph} if
$E \subseteq \{ \{i, i+1\} : i < n-1\}$.
\end{definition}

A vector graph of size $n$ is just a representation of a $\{0,1\}$-value vector of size $n-1$. The choice of a graph representation is for uniformity with the later sections. Given $n \geq 1$, we let $\Vcal_n$ be the set of all vector graphs of size~$n$. In particular, $|\Vcal_n| = 2^{n-1}$. We shall actually consider functions $\mu : \omega \to \omega^{+}$, where $\omega^{+}$ is the successor ordinal of $\omega$.

\begin{definition}
Let $\mu : \omega \to \omega^{+}$ be a function. For every $n \geq 1$ and $D = \{x_0 < \dots < x_{n-1}\} \in [\omega]^n$, let $\Vcal_n(\mu, D)$ be the graph $\Gcal = (\{0, \dots, n-1\}, E)$ such that for each $i < n-1$, $\{i, i+1\} \in E$ if and only if $\mu(x_i) \leq x_{i+1}$.
\end{definition}

An interval $[x, y]$ is \emph{$\mu$-large} if $\mu(x) \leq y$. Otherwise, it is \emph{$\mu$-small}. 
Given a finite set $D = \{x_0 < \dots < x_{n-1}\}$, $\Vcal_n(\mu, D)$ is supposed to code the whole $\mu$-largeness information over $D$. Actually, $\Vcal_n(\mu, D)$ contains only the information about the adjacent intervals, namely, intervals of the form $[x_i, x_{i+1}]$. What about the largeness information about non-adjacent ones? For example, if $[x, y]$ and $[y, z]$ are both $\mu$-small, is $[x, z]$ $\mu$-small or $\mu$-large? We shall see through the notion of $\mu$-transitivity that we can always restrict ourselves to sets over which the whole information of $\mu$-largeness is already fully specified by the $\mu$-largeness information on the adjacent intervals.
A set $H \subseteq \omega$ is \emph{$\mu$-transitive} if for every $x < y < z \in H$, $\mu(x) > y$ and $\mu(y) > z$ if and only if $\mu(x) > z$.

\begin{lemma}\label{lem:mu-transitive-yields-vector-graph}
Fix $\mu : \omega \to \omega^{+}$ and $\rho : \omega \to \omega^{+}$.
Given $n \geq 1$, let $E = \{x_0 < \dots < x_{n-1} \}$ be a $\mu$-transitive set, and $F = \{y_0 < \dots < y_{n-1} \}$ be a $\rho$-transitive set such that $\Vcal_n(\mu, E) = \Vcal_n(\rho, F)$. Then for every $i < j < n$, $[x_i, x_j]$ is $\mu$-large if and only if $[y_i, y_j]$ is $\rho$-large.
\end{lemma}
\begin{proof}
We prove by induction over $m \geq 1$ that for every $i$, $[x_i, x_{i+m}]$ is $\mu$-large if and only if $[y_i, y_{i+m}]$ is $\rho$-large. The base case $m = 1$ is the lemma hypothesis. Suppose it holds up to $m$. Since $E$ is $\mu$-transitive, then $[x_i, x_{i+m+1}]$ is $\mu$-large if and only if either $[x_i, x_{i+m}]$ or $[x_{i+m}, x_{i+m+1}]$ is $\mu$-large. By induction hypothesis, this holds if and only if either $[y_i, y_{i+m}]$ or $[y_{i+m}, y_{i+m+1}]$ is $\rho$-large. Since $F$ is $\rho$-transitive, this holds if and only if $[y_i, y_{i+m+1}]$ is $\rho$-large.
\end{proof}

Cholak and Patey~\cite[Theorem 3.2]{Cholak2019Thin} proved the following theorem. Note that $\mu$ ranges over $\omega$ and not $\omega^{+}$.

\begin{theorem}[Cholak and Patey~\cite{Cholak2019Thin}]\label{thm:avoid-colors-dominate-mu}
Let $\mu : \omega \to \omega$ be a function. For every $n \geq 1$, define $f_n : D \mapsto  \Vcal_n(\mu, D)$. For every infinite set $H \subseteq \omega$ such that $\Vcal_n \not \subseteq f_n[H]^n$, $H$ computes a function dominating~$\mu$. 
\end{theorem}

In particular, this proves that for every hyperarithmetical set~$A$, there is an   $\rt^n_{<\infty, 2^{n-1}-1}$-instance such that every solution computes~$A$. Moreover, this coding technique optimal for non-arithmetical sets from the viewpoint of the number of colors, in the sense that if $A$ is non-arithmetical, then every $\rt^n_{<\infty, 2^{n-1}}$-instance admits a solution which does not compute $A$.

We now prove that this optimality is not only about the number of colors, but also on the nature of the coding, by proving that for every coloring $f : [\omega]^n \to k$ and every non-arithmetical set $A$, then there exists an infinite subdomain $H \subseteq \omega$ such that $A \not \leq_T H$, and over which $f \uh [H]^n$ behaves exactly like our function $D \mapsto \Vcal_n(\mu, D)$, up to a renaming of the colors.

\begin{statement}
$\arithscart^n_k$: For every function $f : [\omega]^n \to k$,
there is a function $\mu : \omega \to \omega^{+}$, an infinite $\mu$-transitive set $H \subseteq \omega$, and a coloring $\chi : \Vcal_n \to k$ such that for every $D \in [H]^n$, $f(D) = \chi(\Vcal_n(\mu, D))$.
\end{statement}

For $n = 1$, there is only one vector graph of size~1, namely $\Gcal = (\{0\}, \emptyset\}$. Therefore $|\Vcal_1| = 1$, and $\arithscart^1_k$ states the existence, for every function $f : \omega \to k$, of a unique color $i < k$ such that for every $x \in H$, $f(x) = i$. Thus $\arithscart^1_k$ is $\rt^1_k$.
The case $n = 2$ yields a new principle.

\begin{statement}\label{stat:explicitlarge}
$\explicitlarge_k$: For every coloring $f : [\omega]^2 \to k$, there are some colors $i_s,i_\ell < k$ and an infinite set $H \subseteq \omega$ such that $f[H]^2 \subseteq \{i_s,i_\ell\}$ and for every $x < y < z \in H$, $f(x, y) = f(y, z) = i_s$ if and only if $f(x, z) = i_s$.
\end{statement}

Intuitively, $i_s$ and $i_\ell$ are the colors of small and large intervals, respectively.

\begin{lemma}
$\arithscart^2_k$ is the statement $\explicitlarge_k$.
\end{lemma}
\begin{proof}
We first prove that $\explicitlarge_k \leq_{id} \arithscart^2_k$.
Let $f : [\omega]^2 \to k$ be a coloring, and let $H$ be an infinite $\arithscart^2_k$-solution. By definition of $\arithscart^2_k$, there is some function $\mu : \omega \to \omega^{+}$ and a coloring $\chi : \Vcal_2 \to k$ such that $H$ is $\mu$-transitive, and for every $x < y \in H$, $f(x, y) = \chi(\Vcal_2(\mu, \{x,y\}))$. Let $\Gcal_0 = (\{0,1\}, \emptyset\}$ and $\Gcal_1 = (\{0,1\}, \{\{0,1\}\})$. In particular, $\Vcal_2 = \{\Gcal_0, \Gcal_1\}$. Let $i_s = \chi(\Gcal_0)$ and $i_\ell = \chi(\Gcal_1)$.
	Since for every $x < y \in H$, $f(x, y) = \chi(\Vcal_2(\mu, \{x,y\}))$, $f[H]^2 \subseteq \{i_s, i_\ell\}$.
	We claim that for every $x < y < z \in H$, $f(x, y) = f(y, z) = i_s$ if and only if $f(x, z) = i_s$. If $i_s = i_\ell$, then $H$ is $f$-homogeneous for color~$i_s$, and satisfies the property, so suppose $i_s \neq i_\ell$. In other words, $\chi$ is one-to-one.
		Fix $x < y < z \in H$. Then $f(x, y) = f(y, z) = i_s$ if and only if $\Vcal_2(\mu, \{x,y\}) = \Vcal_2(\mu, \{y,z\}) = \Gcal_0$, if and only if $\mu(x) > y$ and $\mu(y) > z$. By $\mu$-transitivity of $H$, this holds if and only if $\mu(x) > z$, hence $\Vcal_2(\mu, \{x,z\}) = \Gcal_0$. Since $\chi$ is one-to-one, this holds if and only if $f(x, z) = \chi(\Vcal_2(\mu, \{x,z\})) = \chi(\Gcal_0) = i_s$.

We now prove that $\arithscart^2_k \leq_{id} \explicitlarge_k$.
Let $f : [\omega]^2 \to k$ be a coloring, and let $H \subseteq \omega$ and $i_s, i_\ell < 2$ be such that $f[H]^2 \subseteq \{i_s, i_\ell\}$ and for every $x < y < z \in H$, $f(x, y) = f(y, z) = i_s$ if and only if $f(x, z) = i_s$. Let $\chi(\Gcal_0) = i_s$ and $\chi(\Gcal_1) = i_\ell$.
For every $x \in \omega$, let $\mu(x)$ be either $\min \{ y > x : y \in H \wedge f(x, y) = i_\ell \}$ if it exists, and $\mu(x) = \omega$ otherwise. 
We first claim that $H$ is $\mu$-transitive. Indeed, for every $x < y < z \in H$, $\mu(x) > y$ and $\mu(y) > z$ if and only if $f(x, y) = f(y, z) = i_s$. Since $H$ is a $\Psf$-solution to~$f$, this holds if and only if $f(x, z) = i_s$, hence if and only if $\mu(x) > z$. 
Last, we claim that $H$ is an $\arithscart^2_k$-solution to $f$ with witness $\chi$ and $\mu$. Fix $x < y \in H$. Then $f(x, y) = i_s$ if and only if $\mu(x) > y$, if and only if $\Vcal_2(\mu, \{x,y\}) = \Gcal_0$, if and only if $\chi(\Vcal_2(\mu, \{x,y\})) = \chi(\Gcal_0) = i_s$. Thus $f(x, y) = \chi(\Vcal_2(\mu, \{x,y\}))$.
\end{proof}

By compactness, if we don't consider $\mu$ and $\chi$ as part of the solution , then $\arithscart^n_k$ can be seen as an $\rt^n_k$-like problem. 

\begin{lemma}
There is a c.e.\ set of $\rt^n_k$-like patterns $W$ such that $\arithscart^n_k$ is the problem~$\rt^n_k(W)$.
\end{lemma}
\begin{proof}
Fix a coloring $f : [\omega]^n \to k$.
By compactness, a set $H \subseteq \omega$ is an $\arithscart^n_k$-solution if and only if for every finite set $F \subseteq H$, there is a function $\mu : F \to \omega^{+}$ and a coloring $\chi : \Vcal_n \to k$ such that $F$ is $\mu$-transitive and for every $D \in [F]^n$, $f(D) = \chi(\Vcal_n(\mu, D))$.
Let $W$ be the set of all $\rt^n_k$-patterns such that the above property does not hold. Then $\arithscart^n_k$ is the statement $\rt^n_k(W)$.
\end{proof}

We say that a problem $\Psf$ admits \emph{strong cone avoidance for non-arithmetical cones} if for every set $Z$, every non-$Z$-arithmetical set $C$, and every $\Psf$-instance $X$, there is a $\Psf$-solution $Y$ such that $C \not \leq_T Z \oplus Y$.
 
\begin{theorem}\label{thm:arithscart-sca-nonarith}
$\arithscart^n_k$ admits strong cone avoidance for non-arithmetical cones.
\end{theorem}
\begin{proof}
Fix a set $Z$, a non-$Z$-arithmetical set $C$ and a coloring $f : [\omega]^n \to k$.

Suppose first that $C$ is not $Z$-hyperarithmetical. By Solovay~\cite{Solovay1978Hyperarithmetically}, $C$ is not computably encodable relative to~$Z$. Since for every infinite set $X \subseteq \omega$, there is an infinite $\arithscart^n_k$-solution $Y \subseteq X$ to $f$, there is an $\arithscart^n_k$-solution $H$ to $f$ such that $C \not \leq_T Z \oplus H$. 

Suppose now that $C$ is $Z$-hyperarithmetical. By Groszek and Slaman~\cite{Groszek2007Moduli}, there is a modulus $\mu : \omega \to \omega$ relative to $Z$, that is, for every function $g$ dominating $\mu$, $C \leq_T Z \oplus g$.
Let $f_1$ be defined for each $D \in [\omega]^n$ by $f_1(D) = \langle f(D), \Vcal_n(\mu, D) \rangle$.
By strong cone avoidance of $\rt^n_{<\infty, 2^{n-1}}$ for non-arithmetical cones (see Cholak and Patey~\cite[Theorem 4.15]{Cholak2019Thin}), there is an infinite set $H \subseteq X$ such that $C \not \leq_T Z \oplus H$ and $|f_1[H]^n| \leq 2^{n-1}$. In particular, $H \oplus Z$ does not compute a function dominating $\mu$, so by Theorem~\ref{thm:avoid-colors-dominate-mu}, for every $\Gcal \in \Vcal_n$, there is some $i < k$ and some $D \in [H]^n$ such that $f_1(D) = \langle i, \Gcal \rangle$. Since $|\Vcal_n| = 2^{n-1}$, this $i$ is unique. For each $\Gcal \in \Vcal_n$, let $\chi(\Gcal)$ be this unique~$i$.

We claim that for every $D \in [H]^n$, $f(D) = \chi(\Vcal_n(\mu, D))$.
By definition of $\chi$, $f_1(D) = \langle f(D), \Vcal_n(\mu, D) \rangle = \langle \chi(\Vcal_n(\mu, D)), \Vcal_n(\mu, D) \rangle$. It follows that $f(D) = \chi(\Vcal_n(\mu, D))$. By Cholak and Patey~\cite[Corollary 5.5]{Cholak2019Thin}, there is an infinite $\mu$-transitive subset $H_1 \subseteq H$ such that $C \not \leq_T Z \oplus H_1$. Therefore, $H_1$ is an $\arithscart^n_k$-solution to~$f$.
\end{proof}

The following technical lemma will be useful for Theorem~\ref{thm:not-leq-arithscart-not-sca}, Lemma~\ref{lem:rt-v-w-sca-arith} and Lemma~\ref{lem:w-not-fixed-arithscart-but-v-does}. Note that $\mu_0$ ranges over $\omega^{+}$ while $\mu_1$ ranges over $\omega$.

\begin{lemma}\label{lem:arithscart-avoid-pattern-dominates-mu}
Fix $\chi : \Vcal_n \to k$, and let $f_0 : [\omega]^n \to k$ and $f_1 : [\omega]^n \to k$ be two colorings.
Let $H_0$ be an infinite $\arithscart^n_k$-solution to~$f_0$ with witnesses $\chi$ and $\mu_0 : \omega \to \omega^{+}$, and let $H_1$ be an infinite $\arithscart^n_k$-solution to~$f_1$ with witnesses $\chi$ and $\mu_1 : \omega \to \omega$.
If $H_0$ $f_0$-meets some $\rt^n_k$-pattern $P$ but $H_1$ $f_1$-avoids $P$, then $H_1$ computes a function dominating~$\mu_1$.
\end{lemma}
\begin{proof}
Since $H_0$ $f_0$-meets some $\rt^n_k$-pattern $P$, there is a finite set $E_0 = \{x_0 < \dots < x_{r-1} \} \subseteq H_0$ which $f_0$-meets $P$.
Suppose for the sake of contradiction that $H_1$ does not compute a function dominating $\mu_1$.
By Theorem~\ref{thm:avoid-colors-dominate-mu}, there is a finite set $E_1 = \{y_0 < \dots < y_{r-1}\} \subseteq H_1$ of size $r$ such that $\Vcal_r(\mu_0, E_0) = \Vcal_r(\mu_1, E_1)$. By Lemma~\ref{lem:mu-transitive-yields-vector-graph}, since $H_0$ is $\mu_0$-transitive and $H_1$ is $\mu_1$-transitive, for every $I \in [r]^n$, letting $D_0 = \{x_i : i \in I\}$ and $D_1 = \{y_i : i \in I\}$, $\Vcal_n(\mu_0, D_0) = \Vcal_n(\mu_1, D_1)$. Since $f_0(D_0) = \chi(\Vcal_n(\mu_0, D_0))$ and $f_1(D_1) = \chi(\Vcal_n(\mu_1, D_1))$, then $f_0(D_0) = f_1(D_1)$. Thus $f_0 \uh [E_0]^n$ and  $f_1 \uh [E_1]^n$ have the same function graph.
It follows that $E_1$ $f_1$-meets $P$, so $H_1$ $f_1$-meets $P$. Contradiction.
\end{proof}

\begin{theorem}\label{thm:not-leq-arithscart-not-sca}
Let $W$ be a collection of $\rt^n_k$-patterns such that $\rt^n_k(W) \not \leq_{id} \arithscart^n_k$. Then for every function $\mu : \omega \to \omega$, there is an $\rt^n_k(W)$-instance such that every solution computes a function dominating~$\mu$.
\end{theorem}
\begin{proof}
Since $\rt^n_k(W) \not \leq_{id} \arithscart^n_k$, there is a coloring $f_{fail} : [\omega]^n \to k$ and an $\arithscart^n_k$-solution $H_{fail}$ to $f_{fail}$ witnessed by a function $\mu_{fail} : \omega \to \omega^{+}$ and a coloring $\chi : \Vcal_n \to k$, and such that $H_{fail}$ meets some $\rt^n_k$-pattern $P \in W$.
Let $\mu : \omega \to \omega$ be a function. Let $f : [\omega]^n \to k$ be an instance of $\rt^n_k(W)$ defined by $f(D) = \chi(\Vcal_n(\mu, D))$. 
We claim that every $\rt^n_k(W)$-solution $H \subseteq \omega$ to $f$ computes a function dominating~$\mu$. Let $H$ be an $\rt^n_k(W)$-solution to $f$. In particular, $H$ $f$-avoids $P$. If $H$ does not compute a function dominating $\mu$, then by Cholak and Patey~\cite[Theorem 5.11]{Cholak2019Thin}, there is an infinite $\mu$-transitive subset $H_1 \subseteq H$ such that $H_1$ does not compute a function dominating~$\mu$.
	In particular, $H_1$ is an infinite $\arithscart^n_k$-solution to $f$ with witnesses $\chi$ and $\mu$, and such that $H_1$ $f$-avoids $P$. By Lemma~\ref{lem:arithscart-avoid-pattern-dominates-mu}, $H$ computes a function dominating~$\mu$, contradiction.
\end{proof}

Actually, $\arithscart^n_k$ is the strongest $\rt^n_k$-like problem which admits this avoidance property. The following theorem therefore provides a simple criterion to decide whether an $\rt^n_k$-like problem admits strong cone avoidance for non-arithmetical cones.

\begin{theorem}\label{thm:arithscart-sca-characterization}
A problem $\rt^n_k(W)$ admits strong cone avoidance for non-arithmetical cones if and only if $\rt^n_k(W) \leq_{id} \arithscart^n_k$.
\end{theorem}
\begin{proof}
$\Leftarrow$: Suppose $\rt^n_k(W) \leq_{id} \arithscart^n_k$. Fix a set $Z$, a non-$Z$-arithmetical set $C$ and a coloring $f : [\omega]^n \to k$. By Theorem~\ref{thm:arithscart-sca-nonarith}, there is an $\arithscart^n_k$-solution $H$ to $f$ such that $C \not \leq_T Z \oplus H$. In particular, $H$ is an $\rt^n_k(W)$-solution to~$f$.

$\Rightarrow$: Suppose $\rt^n_k(W) \not\leq_{id} \arithscart^n_k$. Let $\mu : \omega \to \omega$ be a modulus of some non-arithmetical set $C$. By Theorem~\ref{thm:not-leq-arithscart-not-sca}, there is a $\rt^n_k(W)$-instance such that every solution computes a function dominating $\mu$, hence computes~$C$. Therefore $\rt^n_k(W)$ does not admit strong cone avoidance for non-arithmetical cones.
\end{proof}

\subsection{Strongly avoiding non-computable cones}
Whenever the modulus $\mu : \omega \to \omega$ is left-c.e., that is, there is a uniformly computable sequence of functions $\mu_0 \leq \mu_1 \leq \dots$ pointwise limiting to $\mu$, one can exploit more information to compute functions dominating $\mu$. For example, let $f : [\omega]^3 \to 2$ be defined by $f(x, y, z) = \langle b_0, b_1, b_2 \rangle$, where $b_0 = 1$ if and only if $\mu(x) \leq y$, $b_1 = 1$ if and only if $\mu(y) \leq z$, and $b_2 = 1$ if and only if $\mu_z(x) \leq y$. Cholak and Patey~\cite[Theorem 3.17]{Cholak2019Thin} proved that every infinite set $H \subseteq \omega$ such that $|f[H]^3| \leq 4$ computes a function dominating~$\mu$.
We refine again their analysis to obtain a maximal $\rt^n_k$-like principle admitting strong cone avoidance.

\begin{definition}[Cholak and Patey~\cite{Cholak2019Thin}]\label{def:largeness-graph}
A \emph{largeness graph} of size $n$ is a graph $(\{0, \dots, n-1\}, E)$ such that
\begin{itemize}
	\item[(a)] If $\{i,i+1\} \in E$, then for every $j > i+1$, $\{i, j\} \not \in E$
	\item[(b)] If $i < j < n$, $\{i,i+1\} \not \in E$ and $\{j,j+1\} \in E$, then $\{i, j+1\} \in E$
	\item[(c)] If $i +1 < j < n - 1$ and $\{i,j\} \in E$, then $\{i, j+1\} \in E$
	\item[(d)] If $i+1 < j < k < n$ and $\{i, j\} \not \in E$ but $\{i,k\} \in E$, then $\{j-1, k\} \in E$
\end{itemize}
\end{definition}

Let us explain the intuition behind the definition of a largeness graph. Given a left-c.e.\ function $\mu : \omega \to \omega$ and a finite $\mu$-transitive set $D = \{x_0 < x_1 < \dots < x_{n-1}\}$, one can define a graph $\Gcal = (\{0, \dots, n-1\}, E)$ which will represent the whole information about the largeness of the intervals $[x_i, x_j]$ for each $i < j < n$. Since $D$ is $\mu$-transitive, the information about $\mu$-largeness of the intervals $[x_i, x_j]$ for every $i < j$ is already fully specified by the information about $\mu$-largeness of $[x_i, x_{i+1}]$. Therefore, this information is coded only into the adjacent edges. We then set $\{i,i+1\} \in E$ if and only if $[x_i, x_{i+1}]$ is $\mu$-large.

Whenever $[x_i, x_{i+1}]$ is $\mu$-small, this is witnessed after a finite approximation stage $\mu_s$, that is, $[x_i, x_{i+1}]$ is $\mu_s$-small for all but finitely many $s \in \omega$. We can therefore code in $\Gcal$ the information whether some $x_j$ (with $j > i+1$) is large enough to witness this fact, in the sense that $\mu_{x_j}(x_i) > x_{i+1}$. This information is coded into the edges $\{i,j\}$ with $i+1 < j$. We then set $\{i,j\} \in E$ if $[x_i, x_{i+1}]$ is $\mu_{x_j}$-small. Although this coding does not seem consistent with the adjacent edges since the existence of an adjacent edge gives a largeness information, one should really see $\mu_{x_j}$-smallness of $[x_i, x_{i+1}]$ as an information of how big the number $x_j$ is and not on how small the interval $[x_i, x_{i+1}]$ is.

Let us now look at the properties (a) to (d), successively. Property (a) says that if the interval $[x_i, x_{i+1}]$ is $\mu$-large, then it is never $\mu_{x_j}$-small. This property is ensured by construction.
Property (b) says that if $[x_i, x_{i+1}]$ is $\mu$-small and $[x_j, x_{j+1}]$ is $\mu$-large, then the approximation time $x_{j+1}$ is large enough to witness the smallness of $[x_i, x_{i+1}]$. In other words, $[x_i, x_{i+1}]$ is $g_{x_{j+1}}$-small. This property is not structurally ensured, and must be obtained by some extra assumptions on the function $\mu$.
Property (c) simply says that if $[x_i, x_{i+1}]$ is $\mu_{x_j}$-small, then it is $\mu_{x_{j+1}}$-small. This property is structurally ensured by the fact that $\mu_{x_j} \leq \mu_{x_{j+1}}$, which is true of all left-c.e.\ approximations.
Property (d) says that if $x_k$ is large enough to witness the $\mu$-smallness of $[x_i,x_{i+1}]$, but that $x_j$ was not large enough to witness it, then not only by property (b), $[x_{j-1}, x_j]$ cannot be $\mu$-large, but furthermore $x_k$ is large enough to witness the smallness of $[x_{j-1}, x_j]$. In some sense, Property (d) is a refinement of property (b). This property must be also ensured by some extra assumptions on~$\mu$.

Let $\Lcal_n$ be the set of all largeness graphs of size~$n$. Recall that $C_n$ is the $n$th Catalan number. Cholak and Patey~\cite[Lemma 3.16]{Cholak2019Thin} proved that for every $n \geq 1$, $|\Lcal_n| = C_n$.

\begin{definition}
Given a relative left-c.e.\ function $\mu : \omega \to \omega^{+}$ with approximations $\mu_0, \mu_1, \dots$ and a set $D = \{x_0, \dots, x_{n-1}\} \subseteq \omega$,
let $\Lcal_n(\mu, D)$ be the graph $(\{0, \dots, n-1\}, E)$ where
$E = \{ \{p, q\} : \mu_{x_q}(x_p) > x_{p+1} \wedge p+1 < q < n \} \cup \{ \{p, p+1\} : \mu(x_p) \leq x_{p+1} \}$. 
\end{definition}

Note that $\Lcal_n(\mu, D)$ is not a largeness graph in general, because the properties (b) and (d) are not structurally satisfied. We give a sufficient property to ensure $\Lcal_n(\mu, D) \in \Lcal_n$. 

\begin{definition}\label{def:normalized}
A relative left-c.e.\ function $\mu : \omega \to \omega$ is \emph{strongly increasing}
if for every $s \in \omega$ and $x < y \in \omega$, $\mu_s(x) \leq \mu_s(y)$, and if $\mu_{s+1}(x) > \mu_s(x)$ then $\mu_{s+1}(y) > s$. 
\end{definition}



\begin{lemma}\label{lem:strongly-increasing-yields-largeness-graph}
Fix a strongly increasing relative left-c.e.\ function $\mu : \omega \to \omega^{+}$.
For every $n \geq 1$ and $D \in [\omega]^n$, $\Lcal_n(\mu, D) \in \Lcal_n$.
\end{lemma}
\begin{proof}
\emph{Claim 1: For every $w < x < y < z$, if $\mu_y(w) \leq x$ and $\mu_z(w) > x$ then $\mu_z(x) > y$.} Let $s \in [y, z)$ be such that $\mu_s(w) \leq x$ and $\mu_{s+1} > x$. Then since $\mu$ is strongly increasing, $\mu_{s+1}(x) > s \geq y$. In particular, $\mu_z(x) \geq y$. This proves Claim 1.

\emph{Claim 2: For every $w < x < y$, if $\mu(w) > x$ and $\mu(x) \leq y$, then $\mu_y(w) > x$.}
Suppose that Claim 2 does not hold. Then there is some $w < x < y$ such that 
$\mu(w) > x$, $\mu(x) \leq y$ and $\mu_y(w) \leq x$. Then there is a stage $z > y$ such that $\mu_z(w) > x$ and $\mu_z(x) \leq y$. Since $\mu_y(w) \leq x$ and $\mu_z(w) > x$, but  $\mu_z(x) \leq y$, we contradict Claim 1. This proves Claim 2.

Fix $n \geq 1$ and $D = \{x_0 < \dots < x_{n-1}\} \in [\omega]^n$.
We check properties $(a-d)$ of Definition~\ref{def:largeness-graph} for $\Lcal_n(\mu, D) = (\{0, \dots, n-1\}, E)$.
(a): If $\{i,i+1\} \in E$, then $\mu(x_i) \leq x_{i+1}$. In particular, for every $j > i+1$, since $\mu_{x_j} \leq \mu$, $\mu_{x_j}(x_i) \leq x_{i+1}$, so $\{i,j\} \not \in E$.
(b): If $i < j < n$ and $\{i,i+1\} \not \in E$ and $\{j,j+1\} \in E$. Then $\mu(x_i) > x_{i+1}$ and $\mu(x_j) \leq x_{j+1}$. Since $\mu$ is increasing, $\mu(x_{i+1}) \leq x_{j+1}$. Then by Claim 2, $\mu_{x_{j+1}}(x_i) > x_{i+1}$. Thus $\{i,j+1\} \in E$.
(c): If $i+1 < j < n-1$ and $\{i,j\} \in E$, then $\mu_{x_j}(x_i) > x_{i+1}$. Since $\mu_{x_j} \leq \mu_{x_{j+1}}$, then $\mu_{x_{j+1}}(x_i) > x_{i+1}$, so $\{i,j+1\} \in E$.
(d): If $i+1 < j < k < n$ and $\{i,j\} \not \in E$ but $\{i,k\} \in E$. Then $\mu_{x_j}(x_i) \leq x_{i+1}$ but $\mu_{x_k}(x_i) > x_{i+1}$. By Claim 1, $\mu_{x_k}(x_{i+1}) > x_j$. Since $\mu_{x_k}$ is increasing, $\mu_{x_k}(x_{j-1}) > x_j$. Thus $\{j-1, k\} \in E$. 
\end{proof}

We now prove that whenever we choose a left-c.e.\ modulus for a set, we can always assume that it is strongly increasing, without loss of generality.

\begin{lemma}\label{lem:leftce-computes-strongly-increasing}
Every left-c.e.\ function $\mu : \omega \to \omega$
is dominated by a strongly increasing left-c.e.\ function $g : \omega \to \omega$.
\end{lemma}
\begin{proof}
Fix $\mu : \omega \to \omega$ with approximations $\mu_0, \mu_1, \dots$.
We define a uniformly computable sequence of functions $g_0 \leq g_1 \leq \dots : \omega \to \omega$ pointwisely, that is, at stage $x$, we define $g_s(x)$ for every $s \in \omega$. During the construction, for each approximation time $s$ and value $x$, we associate a moving threshold $t_{s, x} \in \omega$ which can only increase, starting with $t_{s,x} = 0$. We will ensure that $g_s(x) \geq t_{s,x}$.

At stage $x$, suppose we have defined $g_s(y)$ for every $y < x$ and $s \in \omega$. We define $g_s(x)$ for every $s \in \omega$.
At approximation time $s$, having defined $g_u(x)$ for every $u < s$, we define $g_s(x)$ to be the maximum value among $\{ g_s(y) : y < x\}$, $\{ g_u(x) : u < s\}$, $t_{s,x}$ and $\mu_s(x)$.
If $s > 0$ and $g_s(x) > g_{s-1}(x)$, then set $t_{s, y} = \max(t_{s, y}, s)$ for every $y > x$. 
Then go to the next approximation time $s+1$. 
Once all the approximations $\langle g_s(x) : s \in \omega\rangle$ are defined, then go to the next stage $x+1$.

Claim 1: $g$ is a strongly increasing left-c.e. function.
By making $g_s(x)$ larger than $\{ g_s(y) : y < x\}$, we have ensured that the function $g_s$ is non-decreasing, that is, for every $x < y$, $g_s(x) \leq g_s(y)$. 
By making $g_s(x)$ larger than $\{ g_u(y) : u < s\}$, we have ensured that $g_s \leq g_{s+1}$, hence that $g_0, g_1, \dots$ are left-c.e.\ approximations.
By making $g_s(x)$ larger than $t_{s,x}$, we have ensured that if there is some  $y < x$ such that $g_{s-1}(y) < g_s(y)$, then $g_s(x) > s-1$. This proves Claim 1.

By making $g_s(x)$ larger than $\mu_s(x)$, we have ensured that $g(x) = \lim_s g_s(x) \geq \lim_s \mu_s(x) = \mu(x)$. So $g$ dominates $\mu$.

Claim 2: For every $x$, $g(x) < \omega$. We prove it by induction over $x$. Indeed, by induction hypothesis, there is a time $v \in \omega$ such that for every $y < x$ and $s > v$, $g_s(y) = g_v(y)$. 
For every $s > v$, the maximum value among $\{ g_s(y) : y < x\}$ is bounded by 
maximum value among $\{g(y) : y < x\}$, which is a finite number by induction hypothesis. Moreover, $t_{s,x} = 0$ since only stages $y < x$ at time $s$ can change the threshold $t_{s,x}$. Moreover, $\mu_s(x)$ is bounded by $\mu(x)$ which is again a finite number. Let $m$ be the maximum value among $\{g(y) : y < x\}$, $\{g_u(x) : u < v \}$ and $\mu(x)$. Then $g_s(x) \leq m$ for every $s > v$. It follows that $\lim_s g_s(x) \leq m$.
This completes the proof of Lemma~\ref{lem:leftce-computes-strongly-increasing}.
\end{proof}

By Lemma~\ref{lem:strongly-increasing-yields-largeness-graph}, taking a strongly increasing relative left-c.e.\ modulus ensures that $\Lcal_n(\mu, D)$ is a largeness graph (Definition~\ref{def:largeness-graph}). As seen in Lemma~\ref{lem:mu-transitive-yields-vector-graph}, $\mu$-transitivity ensures that $\mu$-largeness of the intervals over a set $D$ is fully specified by the $\mu$-largeness of its adjacent intervals. However, some information still seems to be missing in $\Lcal_n(\mu, D)$. Indeed, suppose that $[x_i, x_{i+1}]$ and $[x_{i+1}, x_{i+2}]$ are both $\mu_{x_j}$-small, that is, $\{i, j\}$ and $\{i+1, j\}$ are both edges on $\Lcal_n(\mu, D)$. By $\mu$-transitivity, we know that $[x_i, x_{i+2}]$ is $\mu$-small. However, is it $\mu_{x_j}$-small? Thanks to a stronger notion of $\mu$-transitivity, we can ensure that it will always be the case.
	A set $H \subseteq \omega$ is \emph{strongly $\mu$-transitive} if it is $\mu$-transitive, and for every $w < x < y < z \in H$ such that $\mu_z(w) > x$ and $\mu_z(x) > y$, then $\mu_z(w) > y$. In other words, if $z$ witnesses $\mu$-smallness of both $[w, x]$ and $[x, y]$, then it witnesses $\mu$-smallness of $[w, y]$.

\begin{lemma}\label{lem:strongly-mu-transitive-yields-largeness-graph}
Fix two relative left-c.e.\ functions $\mu : \omega \to \omega^{+}$ and $\rho : \omega \to \omega^{+}$.
Given $n \geq 1$, let $E = \{x_0 < \dots < x_{n-1} \}$ be a strongly $\mu$-transitive set, and $F = \{y_0 < \dots < y_{n-1} \}$ be a strongly $\rho$-transitive set such that $\Lcal_n(\mu, E) = \Lcal_n(\rho, F)$. Then for every $i < j < k < n$, $[x_i, x_j]$ is $\mu_k$-large if and only if $[y_i, y_j]$ is $\rho_k$-large.
\end{lemma}
\begin{proof}
We prove by induction over $m \geq 1$ that for every $i, k$ such that $i+m < k$, $[x_i, x_{i+m}]$ is $\mu_k$-large if and only if $[y_i, y_{i+m}]$ is $\rho_k$-large. The base case $m = 1$ is the lemma hypothesis. Suppose it holds up to $m$. Since $E$ is strongly $\mu$-transitive, then $[x_i, x_{i+m+1}]$ is $\mu_k$-large if and only if either $[x_i, x_{i+m}]$ or $[x_{i+m}, x_{i+m+1}]$ is $\mu_k$-large. By induction hypothesis, this holds if and only if either $[y_i, y_{i+m}]$ or $[y_{i+m}, y_{i+m+1}]$ is $\rho_k$-large. Since $F$ is strongly $\rho$-transitive, this holds if and only if $[y_i, y_{i+m+1}]$ is $\rho_k$-large.
\end{proof}

We now prove that we can computably thin out an infinite $\mu$-transitive set to obtain a strongly $\mu$-transitive infinite subset.

\begin{lemma}\label{lem:transitive-to-strongly-transitive}
Let $\mu : \omega \to \omega^{+}$ be a strongly increasing left-c.e.\ function.
Every infinite $\mu$-transitive set $X \subseteq \omega$ has an infinite $X$-computable strongly $\mu$-transitive subset $Y \subseteq X$.
\end{lemma}
\begin{proof}
Fix $\mu$ and $X$. We build an infinite subsequence $y_0 < y_1 < \dots \in X$ as follows. At stage $0$, let $y_0 = \min X$. At stage $s > 0$, suppose we have defined a strongly $\mu$-transitive finite sequence $y_0 < \dots < y_{s-1}$. Let $y_s > y_{s-1}$ be the least element of $X$ such that for every $i < j < k < s$, if $[y_i, y_j]$ and $[y_j, y_k]$ are $\mu_{y_s}$-small, then $[y_i, y_k]$ is $\mu_{y_s}$-small. Such $y_s$ must be found, since given a fixed tuple $i < j < k < s$ if $[y_i, y_j]$ and $[y_j, y_k]$ are $\mu$-small, then $[y_i, y_k]$ is $\mu$-small, and therefore for all but finitely many $y$, $[y_i, y_k]$ is $\mu_k$-small. Taking a value in $X$ bigger than the max of the thresholds for each tuple $i < j < k < s$, we obtain a $y_s$ with the desired property.
\end{proof}

In the following theorem, note again that $\mu$ ranges over $\omega$ and not $\omega^{+}$.

\begin{theorem}\label{thm:avoid-colors-dominate-left-ce-mu}
Let $\mu : \omega \to \omega$ be a strongly increasing left-c.e.\ function. For every $n \geq 1$, define $f_n : D \mapsto  \Lcal_n(\mu, D)$. For every infinite set $H \subseteq \omega$ such that $\Lcal_n \not \subseteq f_n[H]^n$, $H$ computes a function dominating~$\mu$. 
\end{theorem}
\begin{proof}
Fix $\mu$, $n \geq 1$ and $H$.
By Lemma~\ref{lem:strongly-increasing-yields-largeness-graph}, $f_n[H]^n \subseteq \Lcal_n$. By Cholak and Patey~\cite[Lemma 3.16]{Cholak2019Thin}, $|\Lcal_n| = C_n$. If $\Lcal_n \not \subseteq f_n[H]^n$, then $|f_n[H]^n| < C_n$. Then by Cholak and Patey~\cite[Theorem 3.17]{Cholak2019Thin},  $H$ computes a function dominating~$\mu$.
\end{proof}

In the following statement, note that we state the existence of a \emph{relative} left-c.e.\ function, which is another way of stating the existence of a sequence of functions $\mu_0, \mu_1, \dots$ which satisfy the properties of a left-c.e.\ function independently of its effectiveness.

\begin{statement}
$\largeo^n_k$: For every function $f : [\omega]^n \to k$,
there is a strongly increasing relative left-c.e.\ function $\mu : \omega \to \omega^{+}$, an infinite strongly $\mu$-transitive set $H \subseteq \omega$ and a coloring $\chi : \Lcal_n \to k$ such that for every $D \in [H]^n$, $f(D) = \chi(\Lcal_n(\mu, D))$.
\end{statement}

Again, for $n = 1$, $|\Lcal_n| = 1$ and $\largeo^1_k$ is the statement $\rt^1_k$. Similarly, for $n = 2$, $\Lcal_2 = \Vcal_2$, and $\largeo^2_k$ is the same statement as $\arithscart^2_k$, or equivalently $\explicitlarge_k$, that is, the statement \qt{For every coloring $f : [\omega]^2 \to k$, there is an infinite set $H \subseteq \omega$ and two colors $i_s, i_\ell < k$ such that $f[H]^2 \subseteq \{i_s,i_\ell\}$ and for every $x < y < z \in H$, $f(x, y) = f(y, z) = i_s$ if and only if $f(x, z) = i_s$.}

\begin{theorem}\label{thm:largeo-sca}
$\largeo^n_k$ admits strong cone avoidance.
\end{theorem}
\begin{proof}
Fix two sets $Z$ and $C$ with $C \not \leq_T Z$ and let $f : [\omega]^n \to k$ be an instance of $\largeo^n_k$.

By Lerman~\cite[4.18]{Lerman1983Degrees}, there is a set $Z_1 \geq_T Z$ such that $C$ is $\Delta^0_2(Z_1)$ but $C \not \leq_T Z_1$. Since $C$ is $\Delta^0_2(Z_1)$, there is a left $Z_1$-c.e.\ modulus $\mu : \omega \to \omega$ for $C$. By Lemma~\ref{lem:leftce-computes-strongly-increasing}, we can assume that $\mu$ is strongly increasing.

Let $f_1 : [\omega]^n \to k \times \Lcal_n$ be defined for each $D \in [\omega]^n$ by $f_1(D) = \langle f(D), \Lcal_n(\mu, D) \rangle$.
By strong cone avoidance of $\rt^n_{<\infty, C_n}$ (Cholak and Patey~\cite[Theorem 4.18]{Cholak2019Thin}), there is an infinite set $H \subseteq \omega$ such that $C \not \leq_T Z_1 \oplus H$ and $|f_1[H]^n| \leq C_n$. In particular, $Z_1 \oplus H$ does not compute a function dominating $\mu$, so by Theorem~\ref{thm:avoid-colors-dominate-left-ce-mu}, for every $\Gcal \in \Lcal_n$, there is some $i < k$ and some $D \in [H]^n$ such that $f_1(D) = \langle i, \Gcal \rangle$. Since $|\Lcal_n| = C_n$ (Cholak and Patey~\cite[Lemma 3.16]{Cholak2019Thin}), this $i$ is unique. For each $\Gcal$, let $\chi(\Gcal)$ be such an~$i$. 

We claim that for every $D \in [H]^n$, $f(D) = \chi(\Lcal_n(\mu, D))$.
By definition of $\chi$, $f_1(D) = \langle f(D), \Lcal_n(\mu, D)\rangle = \langle \chi(\Lcal_n(\mu, D)), \Lcal_n(\mu, D)\rangle$. It follows that $f(D) = \chi(\Lcal_n(\mu, D))$. By Cholak and Patey~\cite[Corollary 5.5]{Cholak2019Thin}, there is an infinite $\mu$-transitive subset $H_1 \subseteq H$ such that $C \not \leq_T Z \oplus H_1$. By Lemma~\ref{lem:transitive-to-strongly-transitive}, there is an infinite strongly $\mu$-transitive subset $H_2 \subseteq H_1$ such that $C \not \leq_T Z \oplus H_2$. Therefore, $H_2$ is a $\largeo^n_k$-solution to~$f$.
\end{proof}

The following technical lemma will be useful for Theorem~\ref{thm:not-leq-largeo-not-sca}, Lemma~\ref{lem:rt-v-w-sca} and Lemma~\ref{lem:w-not-fixed-largeo-but-v-does}.

\begin{lemma}\label{lem:largeo-avoid-pattern-dominates-mu}
Fix $\chi : \Lcal_n \to k$, and let $f_0 : [\omega]^n \to k$ and $f_1 : [\omega]^n \to k$ be two colorings.
Let $H_0$ be an infinite $\largeo^n_k$-solution to~$f_0$ with witnesses $\chi$ and $\mu_0 : \omega \to \omega^{+}$, and let $H_1$ be an infinite $\largeo^n_k$-solution to~$f_1$ with witnesses $\chi$ and $\mu_1 : \omega \to \omega$.
If $H_0$ $f_0$-meets some $\rt^n_k$-pattern $P$ but $H_1$ $f_1$-avoids $P$ and $\mu_1$ is left-c.e.\ then $H_1$ computes a function dominating~$\mu_1$.
\end{lemma}
\begin{proof}
Since $H_0$ $f_0$-meets some $\rt^n_k$-pattern $P$, there is a finite set $E_0 = \{x_0 < \dots < x_{r-1} \} \subseteq H_0$ which $f_0$-meets $P$.
Suppose for the sake of contradiction that $H_1$ does not compute a function dominating $\mu_1$.
By Theorem~\ref{thm:avoid-colors-dominate-left-ce-mu}, there is a finite set $E_1 = \{y_0 < \dots < y_{r-1}\} \subseteq H_1$ of size $r$ such that $\Lcal_r(\mu_0, E_0) = \Lcal_r(\mu_1, E_1)$. By Lemma~\ref{lem:mu-transitive-yields-vector-graph} and Lemma~\ref{lem:strongly-mu-transitive-yields-largeness-graph}, since $H_0$ is strongly $\mu_0$-transitive and $H_1$ is strongly $\mu_1$-transitive, for every $I \in [r]^n$, letting $D_0 = \{x_i : i \in I\}$ and $D_1 = \{y_i : i \in I\}$, $\Lcal_n(\mu_0, D_0) = \Lcal_n(\mu_1, D_1)$. Since $f_0(D_0) = \chi(\Lcal_n(\mu_0, D_0))$ and $f_1(D_1) = \chi(\Lcal_n(\mu_1, D_1))$, then $f_0(D_0) = f_1(D_1)$. Thus $f_0 \uh [E_0]^n$ and $f_1 \uh [E_1]^n$ have the same function graph.
It follows that $E_1$ $f_1$-meets $P$, so $H_1$ $f_1$-meets $P$. Contradiction.
\end{proof}

\begin{theorem}\label{thm:not-leq-largeo-not-sca}
Let $W$ be a collection of $\rt^n_k$ patterns such that $\rt^n_k(W) \not \leq_{id} \largeo^n_k$. Then for every left-c.e.\ function $\mu : \omega \to \omega$, there is an $\rt^n_k(W)$-instance such that every solution computes a function dominating~$\mu$.
\end{theorem}
\begin{proof}
Since $\rt^n_k(W) \not \leq_{id} \largeo^n_k$, there is a coloring $f_{fail} : [\omega]^n \to k$ and a $\largeo^n_k$-solution $H_{fail}$ to $f_{fail}$ witnessed by a strongly increasing relative left-c.e.\ function $\mu_{fail}$ and a coloring $\chi : \Lcal_n \to k$, such that $H_{fail}$ is strongly $\mu_{fail}$-transitive, and such that $H_{fail}$ meets some $\rt^n_k$-pattern $P \in W$.
	
Let $\mu : \omega \to \omega$ be a left-c.e.\ function. By Lemma~\ref{lem:leftce-computes-strongly-increasing}, there is a strongly increasing left-c.e.\ function $g : \omega \to \omega$ dominating $\mu$. Let $f : [X]^n \to k$ be an instance of $\rt^n_k(W)$ defined by $f(D) = \chi(\Lcal_n(g, D))$. 
We claim that every $\rt^n_k(W)$-solution $H \subseteq X$ to $f$ computes a function dominating~$g$, hence dominating $\mu$. 
Fix $H$ and suppose for the contradiction that $H$ does not compute a function dominating $g$.
By Cholak and Patey~\cite[Theorem 5.11]{Cholak2019Thin} and Lemma~\ref{lem:transitive-to-strongly-transitive}, there is an infinite strongly $g$-transitive subset $H_1 \subseteq H$ such that $H_1$ does not compute a function dominating~$g$.
	In particular, $H_1$ is an infinite $\largeo^n_k$-solution to $f$ with witnesses $\chi$ and $g$, and such that $H_1$ $f$-avoids $P$. By Lemma~\ref{lem:largeo-avoid-pattern-dominates-mu}, $H$ computes a function dominating~$g$, contradiction.
\end{proof}

\begin{theorem}\label{thm:largeo-sca-characterization}
A problem $\rt^n_k(W)$ admits strong cone avoidance if and only if $\rt^n_k(W) \leq_{id} \largeo^n_k$.
\end{theorem}
\begin{proof}
$\Leftarrow$: Suppose $\rt^n_k(W) \leq_{id} \largeo^n_k$. Fix a set $Z$, a non-$Z$-computable set $C$ and a coloring $f : [\omega]^n \to k$. By Theorem~\ref{thm:largeo-sca}, there is an $\largeo^n_k$-solution $H$ to $f$ such that $C \not \leq_T Z \oplus H$. In particular, $H$ is an $\rt^n_k(W)$-solution to~$f$.

$\Rightarrow$: Suppose $\rt^n_k(W) \not\leq_{id} \largeo^n_k$. Let $\mu : \omega \to \omega$ be a left-c.e.\ modulus of $\emptyset'$. By Theorem~\ref{thm:not-leq-largeo-not-sca}, there is a $\rt^n_k(W)$-instance such that every solution computes a function dominating $\mu$, hence computes~$\emptyset'$. Therefore $\rt^n_k(W)$ does not admit strong cone avoidance.
\end{proof}

Interestingly, Theorem~\ref{thm:largeo-sca-characterization} admits multiple abstract consequences of which one might expect to have a more direct proof. However, there does not seem to be any simpler proof of these facts than proving Theorem~\ref{thm:largeo-sca-characterization}.

The following corollary states that the question of strong cone avoidance of a collection of patterns can be reduced to strong cone avoidance of each pattern individually.

\begin{corollary}\label{cor:pointwise-sca-yields-sca}
If $\rt^n_k(\{P\})$ admits strong cone avoidance for every $P \in W$,
then $\rt^n_k(W)$ admits strong cone avoidance.
\end{corollary}
\begin{proof}
If $\rt^n_k(\{P\})$ admits strong cone avoidance for every $P \in W$,
then by Theorem~\ref{thm:largeo-sca-characterization}, $\rt^n_k(\{P\}) \leq_{id} \largeo^n_k$ for every $P \in W$. It follows that $\rt^n_k(W) \leq_{id} \largeo^n_k$, so again by Theorem~\ref{thm:largeo-sca-characterization}, $\rt^n_k(W)$ admits strong cone avoidance. 
\end{proof}

The following corollary gives an external definition of the maximal problem which admits strong cone avoidance. 

\begin{corollary}
Let $W_{n,k} = \bigcup \{W : \rt^n_k(W) \mbox{ admits strong cone avoidance}\}$. Then $\rt^n_k(W_{n,k})$ admits strong cone avoidance.
\end{corollary}
\begin{proof}
By Corollary~\ref{cor:pointwise-sca-yields-sca}, it suffices to prove that $\rt^n_k(\{P\})$ admits strong cone avoidance for every $P \in W_{n,k}$.
Fix some $P \in W_{n,k}$. By definition, $P \in W$ for some collection $W$ such that $\rt^n_k(W)$ admits strong cone avoidance. In particular, $\rt^n_k(\{P\})$ admits strong cone avoidance.
\end{proof}

\subsection{Avoiding non-computable cones} As explained in Section~\ref{sect:enco-by-computable-instances}, there is a deep relation between the combinatorial features of the colorings over $[\omega]^n$ and the computational features the colorings over $[\omega]^{n+1}$. This link is formalized in the case of cone avoidance by Cholak and Patey in~\cite[Theorem 1.5]{Cholak2019Thin}. We prove the existence of a maximal Ramsey-like principle which admits cone avoidance. For this, we must restrict ourselves to a particular type of largeness graphs, namely, packed graphs.

\begin{definition}[Cholak and Patey~\cite{Cholak2019Thin}]
A largeness graph $\Gcal = (\{0, \dots, n-1\}, E)$ is \emph{packed}
if for every $i < n-2$, $\{i,i+1\} \not \in E$.
\end{definition}

Let $\Pcal_n$ be the set of all packed largeness graphs of size~$n$.
The definition of $\Lcal_n(\mu, D)$, and more precisely the adjacent edges, codes some $\mu$-largeness information. However, in a computable setting, we only have access to the left-c.e.\ approximations of $\mu$. This is why we must restrict ourselves to $\mu$-largeness approximations $\Pcal_n(\mu, D)$, which can be computably coded.

\begin{definition}
Given a relative left-c.e.\ function $\mu : \omega \to \omega^{+}$ with approximations $\mu_0, \mu_1, \dots$ and a set $D = \{x_0, \dots, x_{n-1}\} \subseteq \omega$,
let $\Pcal_n(\mu, D)$ be the graph $(\{0, \dots, n-1\}, E)$ where
$E = \{ \{p, q\} : \mu_{x_q}(x_p) > x_{p+1} \wedge p+1 < q < n \}$. 
\end{definition}

The following two lemmas are obtained by the exact same proof as Lemma~\ref{lem:strongly-increasing-yields-largeness-graph} and Lemma~\ref{lem:strongly-mu-transitive-yields-largeness-graph}, respectively.

\begin{lemma}\label{lem:strongly-increasing-yields-packed-largeness-graph}
Fix a strongly increasing relative left-c.e.\ function $\mu : \omega \to \omega^{+}$.
For every $n \geq 1$ and $D \in [\omega]^n$, $\Pcal_n(\mu, D) \in \Pcal_n$.
\end{lemma}

\begin{lemma}\label{lem:strongly-mu-transitive-yields-packed-largeness-graph}
Fix $\mu : \omega \to \omega^{+}$ and $\rho : \omega \to \omega^{+}$.
Given $n \geq 1$, let $E = \{x_0 < \dots < x_{n-1} \}$ be a strongly $\mu$-transitive set, and $F = \{y_0 < \dots < y_{n-1} \}$ be a strongly $\rho$-transitive set such that $\Pcal_n(\mu, E) = \Pcal_n(\rho, F)$. Then for every $i < j < k < n$, $[x_i, x_j]$ is $\mu_k$-large if and only if $[y_i, y_j]$ is $\rho_k$-large.
\end{lemma}

A coloring $f : [\omega]^{n+1} \to k$ is \emph{stable} if for every $D \in [\omega]^n$, $\lim_y f(D \cup \{y\})$ exists. 
Given a set $D = \{x_0 < \dots < x_{n-1}\}$, if we take some $y \in \omega$ sufficiently large, then $\mu_y$ and $\mu$ will coincide over $[D]^2$.
Then, looking at the packed largeness graph $\Pcal_{n+1}(\mu, D \cup \{y\})$, this graph codes exactly the information of the packed largeness graph $\Pcal_n(\mu, D)$ and the $\mu_y$-largeness information over $[D]^2$, which is by assumption the $\mu$-largeness information over $[D]^2$. These two kind of informations are exactly the ones coded by $\Lcal_n(\mu, D)$. Then, $\Pcal_{n+1}(\mu, D \cup \{y\})$ and $\Lcal_n(\mu, D)$ are in one-to-one correspondance.

Let us define explicitly the one-to-one mapping.
Given a packed largeness graph $\Gcal = (\{0, \dots, n\}, E)$ of size $n+1$, let $\Lcal_n(\Gcal)$ be the largeness graph of size $n$ $(\{0, \dots, n-1\}, E_1)$ where $E_1 = \{\{i, j\} \in E : j < n\} \cup \{\{i,i+1\} : \{i, n\} \not \in E\}$.
The following lemma uses this one-to-one correspondance to relate stable colorings over $[\omega]^{n+1}$ to colorings over $[\omega]^n$.

\begin{lemma}\label{lem:lcelargeo-to-largeo}
Let $\mu : \omega \to \omega^{+}$ be a strongly increasing relative left-c.e.\ function. Fix $n \geq 1$ and define $f_{n+1} : D \mapsto  \Pcal_{n+1}(\mu, D)$.
Let $H$ be an infinite strongly $\mu$-transitive set over which $f_{n+1}$ is stable. For every $D \in [H]^n$, letting $\Gcal = \lim_{y \in H} \Pcal_{n+1}(\mu, D  \cup \{y\})$, $\Lcal_n(\mu, D) = \Lcal_n(\Gcal)$.
\end{lemma}
\begin{proof}
Fix $D = \{x_0 < \dots < x_{n-1}\} \in [H]^n$, and let $x_n \in H$ be sufficiently large so that $\Pcal_{n+1}(\mu, D  \cup \{x_n\}) = \lim_{y \in H} \Pcal_{n+1}(\mu, D  \cup \{y\})$. Let $\Gcal = (\{0, \dots, n\}, E)$ be the packed largeness graph $\Pcal_{n+1}(\mu, D  \cup \{x_n\})$ and $(\{0, \dots, n-1\}, E_1) = \Lcal_n(\Gcal)$. 
We first prove that the adjacent edges in $\Lcal_n(\Gcal)$ are exactly the adjacent edges in $\Lcal_n(\mu, D)$.
For every $i < n-1$, by definition of $\Lcal_n(\Gcal)$, $\{i, i+1\} \in E_1$ if and only if $\{i, n\} \not \in E$. By definition of $\Pcal_{n+1}(\mu, D  \cup \{x_n\})$, $\{i, n\} \not\in E$ if and only if $[x_i, x_{i+1}]$ is $\mu_y$-large, hence if and only if  $[x_i, x_{i+1}]$ is $\mu$-large.
We now prove that the non-adjacent adjacent edges in $\Lcal_n(\Gcal)$ are exactly the non-adjacent edges in $\Lcal_n(\mu, D)$.
	For every $i +1 < j < n$, $\{i,j\} \in E_1$ if and only if $\{i, j\} \in E$,
	hence if and only if $[x_i, x_{i+1}]$ is $\mu_{x_j}$-large.
	Therefore $\Lcal_n(\Gcal) = \Lcal_n(\mu, D)$.
\end{proof}

\begin{theorem}\label{thm:avoid-colors-dominate-left-ce-mu-computable}
Let $\mu : \omega \to \omega$ be a strongly increasing left-c.e.\ function. For every $n \geq 1$, define $f_n : D \mapsto  \Pcal_n(\mu, D)$. For every infinite set $H \subseteq \omega$ such that $\Pcal_n \not \subseteq f_n[H]^n$, $H$ computes a function dominating~$\mu$. 
\end{theorem}
\begin{proof}
Fix $\mu$, $n \geq 1$ and $H$. 
For $n = 1$, then $|\Pcal_1| = 1$. For every infinite set $H \subseteq \omega$, $|f_1[H]^1| = 1$, so $\Pcal_1 \subseteq f_1[H]^1$, and the property vacuously holds. 
For $n > 1$. Suppose for the contradiction that $H$ does not compute a function dominating~$\mu$. By Patey~\cite[Theorem 12]{Patey2017Iterative}, there is an infinite subset $H_1 \subseteq H$ over which $f_n$ is stable and such that $H_1$ does not compute a function dominating~$\mu$. Let $\tilde{f} : [\omega]^{n-1} \to \Pcal_n$ be defined by $\tilde{f}(D) = \lim_{y \in H_1}f_n(D \cup \{y\}) = \lim_{y \in H_1} \Pcal_n(\mu, \cup \{y\})$. By Lemma~\ref{lem:lcelargeo-to-largeo}, for every $D \in [H_1]^{n-1}$, $\Lcal_{n-1}(\mu, D) = \Lcal_n(\tilde{f}(D))$.

By Lemma~\ref{lem:strongly-increasing-yields-packed-largeness-graph}, $f_n[H_1]^n \subseteq \Pcal_n$. By Cholak and Patey~\cite[Lemma 3.15,Lemma 3.16]{Cholak2019Thin}, $|\Pcal_n| = C_{n-1}$. Since $\Pcal_n \not \subseteq f_n[H]^n$, then $|f_n[H]^n| < C_{n-1}$. It follows that, letting $g_{n-1} : D \mapsto \Lcal_{n-1}(\mu, D)$, $|g_{n-1}[H]^{n-1}| = |f_n[H]^n| < C_{n-1}$.
	By Cholak and Patey~\cite[Lemma 3.16]{Cholak2019Thin}, $|\Lcal-{n-1}| = C_{n-1}$, therefore, $\Lcal_{n-1} \not \subseteq g_{n-1}[H]^{n-1}$, so by Theorem~\ref{thm:avoid-colors-dominate-left-ce-mu}, $H_1$ computes a function dominating~$\mu$. Contradiction.
\end{proof}

\begin{statement}
$\lcelargeo^n_k$: For every function $f : [\omega]^n \to k$,
there is a strongly increasing relative left-c.e.\ function $\mu : \omega \to \omega^{+}$, an infinite strongly $\mu$-transitive set $H \subseteq \omega$ and a coloring $\chi : \Pcal_n \to k$ such that for every $D \in [H]^n$, $f(D) = \chi(\Pcal_n(\mu, D))$.
\end{statement}

In the cases $n = 1$ and $n = 2$, there is exactly one packed largeness graph of size $n$, namely, the graph with no edges. Therefore $|\Pcal_n| = 1$, and $\lcelargeo^1_k$ and $\lcelargeo^2_k$ are exactly $\rt^1_k$ and $\rt^2_k$, respectively. The case $n = 3$ yields a new principle.

\begin{statement}\label{stat:lceexplicitlarge}
$\lceexplicitlarge_k$: ``For every coloring $f : [\omega]^3 \to k$, there are two colors $i_s,i_\ell < k$ and an infinite set $H \subseteq \omega$ such that $f[H]^3 \subseteq \{i_s,i_\ell\}$ and for every $w < x < y < z \in H$,
\begin{itemize}
	\item[(a)] $f(w, x, z) = f(x, y, z) = i_s$ if and only if $f(w, y, z) = i_s$
	\item[(b)] if $f(w, x, y) = i_s$ then $f(w, x, z) = i_s$
	\item[(c)] if $f(w, x, y) = i_\ell$ and $f(w, x, z) = i_s$ then $f(x, y, z) = i_s$.''
\end{itemize}
\end{statement}

Informally, $f(x, y, z) = i_s$ if $[x, y]$ is $\mu_z$-small, and $f(x, y, z) = i_\ell$ otherwise. We now prove that $\lcelargeo^3_k$ and $\lceexplicitlarge_k$ are the same statement by bi-reduction.

\begin{lemma}
$\lceexplicitlarge_k \leq_{id} \lcelargeo^3_k$.
\end{lemma}
\begin{proof}
Let $f : [\omega]^3 \to k$ be a coloring, and let $H$ be an infinite $\lcelargeo^3_k$-solution to~$f$. By definition of $\lcelargeo^3_k$, there is a strongly increasing relative left-c.e.\ modulus $\mu : \omega \to \omega^{+}$ and a function $\chi : \Pcal_3 \to k$ such that $H$ is strongly $\mu$-transitive, and for every $D \in [H]^3$, $f(D) = \chi(\Pcal_3(\mu,D))$. Let $\Gcal_0 = (\{0,1,2\}, \emptyset)$ and $\Gcal_1 = (\{0,1,2\}, \{\{0,2\}\})$. In particular, $\Pcal_3 = \{\Gcal_0, \Gcal_1\}$. Let $i_s = \chi(\Gcal_1)$ and $i_\ell = \chi(\Gcal_0)$. We have $f[H]^3 \subseteq \{i_s, i_\ell\}$. If $i_s = i_\ell$, then $H$ is $f$-homogeneous, and $H$ is an $\lceexplicitlarge_k$-solution to $f$, so suppose $i_s \neq i_\ell$. In particular, $\chi$ is one-to-one. We now prove properties (a-c) of Statement~\ref{stat:lceexplicitlarge}. Fix $w < x < y < z \in H$.

(a): $f(w, x, z) = f(x, y, z) = i_s$ if and only if $\chi(\Pcal_3(\mu, \{w, x, z\})) = \chi(\Pcal_3(\mu, \{x, y, z\})) = i_s$ if and only if $\Pcal_3(\mu, \{w, x, z\}) = \Pcal_3(\mu, \{x, y, z\}) = \Gcal_1$ if and only if $[w, x]$ and $[x, y]$ are $\mu_z$-small. Since $H$ is strongly $\mu$-transitive, this holds if and only if $[w, y]$ is $\mu_z$-small, if and only if $\Pcal_3(\mu, \{w, y, z\}) = \Gcal_1$ if and only if $f(w, y, z) = \chi(\Pcal_3(\mu, \{w, y, z\})) = \chi(\Gcal_1) = i_s$.

(b): If $f(w, x, y) = i_s$, then $\chi(\Pcal_3(\mu, \{w, x, y\})) = i_s$, so $\Pcal_3(\mu, \{w, x, y\}) = \Gcal_1$. In particular, $[w, x]$ is $\mu_y$-small, hence is $\mu_z$-small. Therefore, $\Pcal_3(\mu, \{w, x, z\}) = \Gcal_1$, so $f(w, x, z) = \chi(\Pcal_3(\mu, \{w, x, z\})) = \chi(\Gcal_1) = i_s$.

(c): If $f(w, x, y) = i_\ell$ and $f(w, x, z) = i_s$, then $\chi(\Pcal_3(\mu, \{w, x, y\})) = i_\ell$ and $\chi(\Pcal_3(\mu, \{w, x, z\})) = i_s$. Therefore, $\Pcal_3(\mu, \{w, x, y\}) = \Gcal_0$ and $\Pcal_3(\mu, \{w, x, z\}) = \Gcal_1$. It follows that $[w, x]$ is $\mu_y$-large but $\mu_z$-small. By Claim 1 in the proof of Lemma~\ref{lem:strongly-increasing-yields-largeness-graph}, since $\mu$ is strongly increasing, $[x, y]$ is $\mu_z$-small. Thus $\Pcal_3(\mu, \{x, y, z\}) = \Gcal_1$, so $f(x, y, z) = \chi(\Pcal_3(\mu, \{x, y, z\})) = \chi(\Gcal_1) = i_s$.
\end{proof}

\begin{lemma}
$\lcelargeo^3_k \leq_{id} \lceexplicitlarge_k$.
\end{lemma}
\begin{proof}
Let $f : [\omega]^3 \to k$ be a coloring, and let $H$ be an infinite $\lceexplicitlarge_k$-solution, that is, there are some colors $i_s, i_\ell < k$ such that $f[H]^3 \subseteq \{i_s, i_\ell\}$ and properties (a-c) of Statement~\ref{stat:lceexplicitlarge} hold. 
For every $x, z \in \omega$, let $x_0$ and $z_0$ be the least elements of $H \cap [x, \infty)$ and $H \cap [z, \infty)$, respectively. Let $\mu_z(x)$ be the least element $y_0$ of $H \cap (x_0, z_0)$ such that $f(x_0, y_0, z_0) = i_\ell$ if it exists. Otherwise $\mu_z(x) = z$.

Claim 1: For every $x, z \in \omega$, $\mu_z(x) \leq z$.
Indeed, let $x_0$ and $z_0$ be the least elements of $H \cap [x, \infty)$ and $H \cap [z, \infty)$, respectively. If there is a least element $y_0$ of $H \cap (x_0, z_0)$ such that $f(x_0, y_0, z_0) = i_\ell$, then $y_0 < z_0$ and by definition of $z_0$, $y_0 < z$. It follows that $\mu_z(x) = y_0 < z$. If there is no such $y_0$, then $\mu_z(x) = z \leq z$. This proves Claim 1.

Claim 2: For every $x \in \omega$ and $u \leq v \in \omega$, $\mu_u(x) \leq \mu_v(x)$.
Let $x_0$, $u_0$ and $v_0$ be the least elements of $H \cap [x, \infty)$, $H \cap [u, \infty)$ and $H \cap [v, \infty)$, respectively. In particular, $u_0 \leq v_0$. Case 1: there is no $y_u \in H \cap (x_0, u_0)$ such that $f(x_0, y_u, u_0) = i_\ell$. Then $\mu_u(x) = u$. If there is no $y_v \in H \cap (x_0, v_0)$ such that $f(x_0, y_v, v_0) = i_\ell$, then $\mu_v(x) = v \geq u = \mu_u(x)$, and we are done. So suppose there is such a $y_v \in H \cap (x_0, v_0)$. If $y_v \geq u_0$, then $\mu_v(x) = y_v \geq u_0 \geq u = \mu_u(u)$ and we are done. If $y_v < u_0$, then $f(x_0, y_v, v_0) = i_\ell$ and since $y_u$ does not exist, $f(x_0, y_v, u_0) = i_s$. This contradicts property (b) of Statement~\ref{stat:lceexplicitlarge}. Case 2: $y_u$ exists. Then $y_u < u_0$, so $y_u < u$ and for every $y \in H \cap (x_0, y_u)$, $f(x_0, y, u_0) = i_s$. By property (b) of Statement~\ref{stat:lceexplicitlarge}, for every $y \in H \cap (x_0, y_u)$, $f(x_0, y, v_0) = i_s$. Therefore if $y_v$ exists, then $y_v \geq y_u$ and $\mu_u(x) \leq \mu_v(x)$. If $y_v$ does not exist, then $\mu_v(x) = v \geq u > y_u = \mu_u(x)$. This proves Claim 2.

Claim 3: For every $z \in \omega$ and $u \leq v$, $\mu_z(u) \leq \mu_z(v)$.
Let $u_0, v_0$ and $z_0$ be the least elements of $H \cap [u, \infty)$, $H \cap [v, \infty)$ and $H \cap [z, \infty)$, respectively. In particular, $u_0 \leq v_0$. Case 1: there is no $y_v \in H \cap (v_0, z_0)$ such that $f(v_0, y_v, z_0) = i_\ell$. Then $\mu_v(x) = v$. Since by Claim 1, $\mu_u(x) \leq u \leq v = \mu_v(x)$, we are done. Case 2: there is such a $y_v$. In particular, $f(v_0, y_v, z_0) = i_\ell$. By property (a) of Statement~\ref{stat:lceexplicitlarge}, $f(u_0, y_v, z_0) = i_\ell$, so there is a least $y_u \in H \cap (u_0, z_0)$ such that $f(u_0, y_u, z_0) = i_\ell$ and $y_u \leq y_v$. Then $\mu_z(u) = y_u \leq y_v = \mu_z(v)$. This proves Claim 3.

Claim 4: For every $z \in \omega$ and $u < v \in \omega$, if $\mu_{z+1}(u) > \mu_z(u)$, then $\mu_{z+1}(v) > z$.
Let $z_0, z_1, u_0, v_0$ be the least elements of $H \cap [z, \infty)$,  $H \cap [z+1, \infty)$, $H \cap [u, \infty)$ and $H \cap [v, \infty)$, respectively.
If $z_0 = z_1$, we are done, so suppose $z_0 < z_1$. In particular, $z_1$ is the immediate successor of $z_0$ in $H$. If there is no least $y_v \in H \cap (v_0, z_1)$ such that $f(v_0, y_v, z_1) = i_\ell$, then $\mu_{z+1}(v) = z+1$ and we are done as well. So suppose there is such a $y_v$. By property (a) of Statement~\ref{stat:lceexplicitlarge}, $f(u_0, y_v, z_1) = i_\ell$, so there is a least $y_u \in H \cap (u_0, z_1)$ such that $f(u_0, y_u, z_1) = i_\ell$. Since $z_1$ is the immediate successor of $z_0$ in $H$, either $y_u = z_0$, or $y_u < z_0$. Case 1: $y_u = z_0$. By Claim 3, $\mu_{z+1}(v) \geq \mu_{z+1}(u) = y_u = z_0$. If $z_0 > z$, we are done, so suppose $y_u = z_0 = z$. If $\mu_z(u) = z$, then Claim 4 is vacuously satisfied, so suppose $\mu_z(u) < z$. It follows that there is a least $y \in H \cap (u_0, z_0)$ such that $f(u_0, y, z_0) = i_\ell$. By minimality of $y_0$, $f(u_0, y, z_1) = i_s$. Then by property (c) of Statement~\ref{stat:lceexplicitlarge}, $f(y, z_0, z_1) = i_s$. Since $f(u_0, y_u, z_1) = i_\ell$, and $y_u = z_0 = z$, then $f(y, z, z_1) = i_s$ and $f(u_0, z, z_1) = i_\ell$. So by property (a) of Statement~\ref{stat:lceexplicitlarge}, $f(u_0, y, z_1) = i_\ell$, contradiction.
	Case 2: $y_u < z_0$. Then there is some some $y \in H \cap (u_0, z_0)$ such that $f(u_0, y, z_0) = i_\ell$, with $y < y_u$. By definition of $y_u$ being a least element, $f(u_0, y, z_1) = i_s$. Then by property (c) of Statement~\ref{stat:lceexplicitlarge}, $f(y, z_0, z_1) = i_s$, and since $y < y_u < z_0$, and $f(u_0, y_u, z_1) = i_\ell$, then by property (a) of Statement~\ref{stat:lceexplicitlarge}, $f(u_0, y, z_1) = i_\ell$ or $f(y, y_u, z_1) = i_\ell$. The former case does not hold, so $f(y, y_u, z_1) = i_\ell$. Again by property (a) of Statement~\ref{stat:lceexplicitlarge}, $f(y, z_0, z_1) = i_\ell$. Contradiction. 
	
Claim 5: For $w < x < y < z \in H$ such that $\mu_z(w) > x$ and $\mu_z(x) > y$, then $\mu_z(w) > y$. By definition of $\mu$, $f(w, x, z) = f(x, y, z) = i_s$. By property (a) of Statement~\ref{stat:lceexplicitlarge}, $f(w, y, z) = i_s$. Therefore, $\mu_z(w) > y$. This proves Claim~5.

By Claim 2, 3 and 4, $\mu$ is a strongly increasing left-c.e.\ function.
Moreover, for every $x < y < z \in H$, $f(x, y, z) = i_\ell$ if and only if $\mu_z(x) \leq y$, so $f(x, y, z) = \chi(\Pcal_3(\mu, \{x, y, z\}))$.
By Claim 5, $H$ is strongly $\mu$-transitive. Therefore $H$ is a $\lcelargeo^3_k$-solution to $f$ with witnesses $\chi$ and $\mu$.
\end{proof}

\begin{theorem}\label{thm:lcelargeo-ca}
$\lcelargeo^n_k$ admits cone avoidance.
\end{theorem}
\begin{proof}
Fix two sets $Z$ and $C$ with $C \not \leq_T Z$ and let $f : [\omega]^n \to k$ be a $Z$-computable instance of $\lcelargeo^n_k$.

By Lerman~\cite[4.18]{Lerman1983Degrees}, there is a set $Z_1 \geq_T Z$ such that $C$ is $\Delta^0_2(Z_1)$ but $C \not \leq_T Z_1$. Since $C$ is $\Delta^0_2(Z_1)$, there is a left $Z_1$-c.e.\ modulus $\mu : \omega \to \omega$ for $C$. By Lemma~\ref{lem:leftce-computes-strongly-increasing}, we can assume that $\mu$ is strongly increasing.

Let $f_1 : [\omega]^n \to k \times \Pcal_n$ be defined for each $D \in [\omega]^n$ by $f_1(D) = \langle f(D), \Pcal_n(\mu, D) \rangle$.
By cone avoidance of $\rt^n_{<\infty, C_{n-1}}$ (Cholak and Patey~\cite[Corollary 4.19]{Cholak2019Thin}), there is an infinite set $H \subseteq \omega$ such that $C \not \leq_T Z_1 \oplus H$ and $|f_1[H]^n| \leq C_{n-1}$. In particular, $Z_1 \oplus H$ does not compute a function dominating $\mu$, so by Theorem~\ref{thm:avoid-colors-dominate-left-ce-mu}, for every $\Gcal \in \Pcal_n$, there is some $i < k$ and some $D \in [H]^n$ such that $f_1(D) = \langle i, \Gcal \rangle$. Since $|\Pcal_n| = C_{n-1}$ (Cholak and Patey~\cite[Lemma 3.15,Lemma 3.16]{Cholak2019Thin}), this $i$ is unique. For each $\Gcal$, let $\chi(\Gcal)$ be such an~$i$. 

We claim that for every $D \in [H]^n$, $f(D) = \chi(\Pcal_n(\mu, D))$.
By definition of $\chi$, $f_1(D) = \langle f(D), \Pcal_n(\mu, D)\rangle = \langle \chi(\Pcal_n(\mu, D)), \Pcal_n(\mu, D) \rangle$. It follows that $f(D) = \chi(\Pcal_n(\mu, D))$. By Cholak and Patey~\cite[Corollary 5.5]{Cholak2019Thin}, there is an infinite $\mu$-transitive subset $H_1 \subseteq H$ such that $C \not \leq_T Z \oplus H_1$. By Lemma~\ref{lem:transitive-to-strongly-transitive}, there is an infinite strongly $\mu$-transitive subset $H_2 \subseteq H_1$ such that $C \not \leq_T Z \oplus H_2$. Therefore, $H_2$ is a $\lcelargeo^n_k$-solution to~$f$.
\end{proof}

The following technical lemma will be useful for Theorem~\ref{thm:not-leq-lcelargeo-not-ca}, Lemma~\ref{lem:rt-v-w-ca} and Lemma~\ref{lem:w-not-fixed-lcelargeo-but-v-does}.

\begin{lemma}\label{lem:lcelargeo-avoid-pattern-dominates-mu}
Fix $\chi : \Pcal_n \to k$, and let $f_0 : [\omega]^n \to k$ and $f_1 : [\omega]^n \to k$ be two colorings such that $f_1$ is computable.
Let $H_0$ be an infinite $\lcelargeo^n_k$-solution to~$f_0$ with witnesses $\chi$ and $\mu_0 : \omega \to \omega$, and let $H_1$ be an infinite $\lcelargeo^n_k$-solution to~$f_1$ with witnesses $\chi$ and $\mu_1 : \omega \to \omega$.
If $H_0$ $f_0$-meets some $\rt^n_k$-pattern $P$ but $H_1$ $f_1$-avoids $P$ and $\mu_1$ is left-c.e., then $H_1$ computes a function dominating~$\mu_1$.
\end{lemma}
\begin{proof}
Since $H_0$ $f_0$-meets some $\rt^n_k$-pattern $P$, there is a finite set $E_0 = \{x_0 < \dots < x_{r-1} \} \subseteq H_0$ which $f_0$-meets $P$.
Suppose for the sake of contradiction that $H_1$ does not compute a function dominating $\mu_1$.
By Theorem~\ref{thm:avoid-colors-dominate-left-ce-mu-computable}, there is a finite set $E_1 = \{y_0 < \dots < y_{r-1}\} \subseteq H_1$ of size $r$ such that $\Pcal_r(\mu_0, E_0) = \Pcal_r(\mu_1, E_1)$. By Lemma~\ref{lem:strongly-mu-transitive-yields-packed-largeness-graph}, since $H_0$ is strongly $\mu_0$-transitive and $H_1$ is strongly $\mu_1$-transitive, for every $I \in [r]^n$, letting $D_0 = \{x_i : i \in I\}$ and $D_1 = \{y_i : i \in I\}$, $\Pcal_n(\mu_0, D_0) = \Pcal_n(\mu_1, D_1)$. Since $f_0(D_0) = \chi(\Pcal_n(\mu_0, D_0))$ and $f_1(D_1) = \chi(\Pcal_n(\mu_1, D_1))$, then $f_0(D_0) = f_1(D_1)$. Thus $f_0 \uh [E_0]^n$ and $f_1 \uh [E_1]^n$ have the same function graph.
It follows that $E_1$ $f_1$-meets $P$, so $H_1$ $f_1$-meets $P$. Contradiction.
\end{proof}

\begin{theorem}\label{thm:not-leq-lcelargeo-not-ca}
Let $W$ be a collection of $\rt^n_k$-patterns such that $\rt^n_k(W) \not \leq_{id} \lcelargeo^n_k$. Then for every left-c.e.\ function $\mu : \omega \to \omega$, there is a computable $\rt^n_k(W)$-instance such that every solution computes a function dominating~$\mu$.
\end{theorem}
\begin{proof}
Since $\rt^n_k(W) \not \leq_{id} \lcelargeo^n_k$, there is a coloring $f_{fail} : [\omega]^n \to k$ and a $\lcelargeo^n_k$-solution $H_{fail}$ to $f_{fail}$ witnessed by strongly increasing left-c.e.\ function $\mu_{fail} : \omega \to \omega^{+}$ and a coloring $\chi : \Pcal_n \to k$, such that $H_{fail}$ is strongly $\mu_{fail}$-transitive, and $H_{fail}$ meets some $\rt^n_k$-pattern $P \in W$.
	
Let $\mu : \omega \to \omega$ be a left-c.e.\ function. By Lemma~\ref{lem:leftce-computes-strongly-increasing}, there is a strongly increasing left-c.e.\ function $g : \omega \to \omega$ dominating $\mu$. Let $f : [X]^n \to k$ be a computable instance of $\rt^n_k(W)$ defined by $f(D) = \chi(\Pcal_n(g, D))$. 
We claim that every $\rt^n_k(W)$-solution $H$ to $f$ computes a function dominating~$g$, hence dominating $\mu$. 
Fix $H$ and suppose for the contradiction that $H$ does not compute a function dominating $g$. By Cholak and Patey~\cite[Theorem 5.11]{Cholak2019Thin} and Lemma~\ref{lem:transitive-to-strongly-transitive}, there is an infinite strongly $g$-transitive subset $H_1 \subseteq H$ such that $H_1$ does not compute a function dominating~$g$.
	In particular, $H_1$ is an infinite $\lcelargeo^n_k$-solution to $f$ with witnesses $\chi$ and $g$, and such that $H_1$ $f$-avoids $P$. By Lemma~\ref{lem:lcelargeo-avoid-pattern-dominates-mu}, $H$ computes a function dominating~$g$, contradiction.
\end{proof}

\begin{theorem}\label{thm:lcelargeo-ca-characterization}
A problem $\rt^n_k(W)$ admits cone avoidance if and only if $\rt^n_k(W) \leq_{id} \lcelargeo^n_k$.
\end{theorem}
\begin{proof}
$\Leftarrow$: Suppose $\rt^n_k(W) \leq_{id} \lcelargeo^n_k$. Fix a set $Z$, a non-$Z$-computable set $C$ and a $Z$-computable coloring $f : [\omega]^n \to k$. By Theorem~\ref{thm:lcelargeo-ca}, there is a $\lcelargeo^n_k$-solution $H$ to $f$ such that $C \not \leq_T Z \oplus H$. In particular, $H$ is an $\rt^n_k(W)$-solution to~$f$.

$\Rightarrow$: Suppose $\rt^n_k(W) \not\leq_{id} \lcelargeo^n_k$. Let $\mu : \omega \to \omega$ be a left-c.e.\ modulus of $\emptyset'$. By Theorem~\ref{thm:not-leq-lcelargeo-not-ca}, there is a computable $\rt^n_k(W)$-instance such that every solution computes a function dominating $\mu$, hence computes~$\emptyset'$. Therefore $\rt^n_k(W)$ does not admit cone avoidance.
\end{proof}

\section{Promise Ramsey-like theorems}\label{sect:promise-ramsey-like}

The class of Ramsey-like problems encompasses Ramsey's theorem and the Erd\H{o}s-Moser theorem~\cite{Bovykin2005strength} since both statements are of the form \qt{For every coloring $f : [\omega]^n \to k$, there is an infinite set $H$ which avoids some set of patterns.} There are however two kind of consequences of Ramsey's theorem which do not fit within this framework. 

First, some statements have a restricted class of instances. For example, the Ascending Descending Sequence~\cite{Hirschfeldt2007Combinatorial} ($\ads$) principle asserts that every infinite linear order admits an infinite ascending or descending sequence. A linear order $\Lcal = (\omega, \prec_{\Lcal})$ can be formalized as a coloring $f : [\omega]^2 \to 2$ such that for every $x < y$, $f(x, y) = 1$ if and only if $x \prec_{\Lcal} y$. Such a coloring is called \emph{transitive}, since for every $x < y < z$ and $i < 2$, if $f(x, y) = i$ and $f(y, z) = i$ then $f(x, z) = i$. $\ads$ can therefore be formulated as the statement \qt{For every transitive coloring $f : [\omega]^2 \to 2$, there exists an infinite $f$-homogeneous set.} 

Second, some consequences of Ramsey's theorem such as the Free Set theorem~\cite{Cholak2001Free} involve $\omega$-colorings of $[\omega]^n$. The free set theorem for $n$-tuples is the statement \qt{For every coloring $f : [\omega]^n \to \omega$, there is an infinite set $H$ such that for every $x \in H$, $x \not \in f[H \setminus \{x\}]^n$.} Suppose that for every $D \in [\omega]^n$, $f(D) \leq \min D$. 
	We can define a coloring $g : [\omega]^{n+1} \to 2$ by $g(x_0, x_1, \dots, x_n) = 1$ if and only if $f(x_1, \dots, x_n) = x_0$. 
	Every $g$-homogeneous set must be of color 0, and therefore be a free set solution to $f$.
	Then one can formulate this restricted version of the free set theorem as the statement \qt{For every coloring $g : [\omega]^{n+1} \to \omega$ such that for every $x \in \omega$, there is at most one $D \in [\omega]^n$ such that $g(\{x\} \cup D) = 1$, there is an infinite $g$-homogeneous set.} The full free set theorem can also fit within this framework with additional technicalities.
	
In both examples, the problems can be formulated as statements \qt{For every coloring $f : [\omega]^n \to k$ which avoids some set of patterns, there is an infinite set $H$ which avoids another set of patterns.} In computational complexity, problems whose class of instances is restricted by some properties are known as \emph{promise problems}. This motivates the following definition.

\begin{definition}
Given two collections $V$ and $W$ of $\rt^n_k$-patterns, the \emph{promise $\rt^n_k$-like problem} $\rt^n_k(V, W)$ is the problem whose instances are colorings $f : [\omega]^n \to k$ such that $\omega$ $f$-avoids every pattern in $V$. An $\rt^n_k(V,W)$-solution to an instance $f$ is an infinite set $H \subseteq \omega$ $f$-avoiding every pattern in $W$.
\end{definition}

For example, $\ads$ is the promise $\rt^2_2$-like problem $\rt^2_2(W_{\ads}, W_{\rt^2_2})$ with $W_{\ads} = \{ \ftt(x_0, x_1) = 1 \wedge \ftt(x_1, x_2) = 1 \wedge \ftt(x_0, x_2) = 0, \ftt(x_0, x_1) = 0 \wedge \ftt(x_1, x_2) = 0 \wedge \ftt(x_0, x_2) = 1 \}$. We interpret a function $f : [\omega]^2 \to 2$ such that $\omega$ $f$-avoids the pattern $W_{\ads}$ as a linear order $\prec$ defined by $x \prec y$ if and only if $x <_{\Nb} y$ and $f(x, y) = 1$ or $x >_{\Nb} y$ and $f(x, y) = 0$.

Similarly, $\cac$ is the promise $\rt^2_3$-like problem $\rt^2_3(W_{\cac}, W_{\rt^2_2})$ with $W_{\cac} = W_{\ads}$.
We interpret a function $f : [\omega]^2 \to 3$ such that $\omega$ $f$-avoids the pattern $W_{\ads}$ as a partial order $\prec$ defined by $x \prec y$ if and only if $x <_{\Nb} y$ and $f(x, y) = 1$ or $x >_{\Nb} y$ and $f(x, y) = 0$.

\subsection{Strongly avoiding non-arithmetical cones} We now extend our analysis of strong cone avoidance for non-arithmetical cones to the class of promise Ramsey-like problems. For this, we need to define a restricted version of $\arithscart^n_k$, in which the coloring $\chi : \Vcal_n \to k$ must belong to a predefined set of colorings.

\begin{statement}
Let $\Rcal$ be a set of functions of type $\chi : \Vcal_n \to k$.
$\arithscartp{\Rcal}^n_k$: For every function $f : [\omega]^n \to k$,
there is a function $\mu : \omega \to \omega^{+}$, an infinite $\mu$-transitive set $H \subseteq \omega$, and a coloring $\chi \in \Rcal$ such that for every $D \in [H]^n$, $f(D) = \chi(\Vcal_n(\mu, D))$.
\end{statement}

One particular case of interest is whenever $\Rcal$ is a singleton $\{\chi\}$. The following notion will be very useful in our case analysis.
Given some $n, k \in \omega$ and a collection $W$ of $\rt^n_k$-patterns,
let 
$$
\Rcal^n_k(W) = \left\{\chi : \Vcal_n \to k \mbox{ }|\mbox{ } \rt^n_k(W) \leq_{id} \arithscartp{\{\chi\}}^n_k\right\}
$$
In particular, if $\rt^n_k(W)$ is a true statement, then every constant function $\chi : \Vcal_n \to k$ belongs to~$\Rcal^n_k(W)$.

\begin{lemma}\label{lem:rt-v-w-sca-arith}
Let $V$ and $W$ be collections of $\rt^n_k$-patterns such that
$\rt^n_k(W) \leq_{id} \arithscartp{\Rcal^n_k(V)}^n_k$. 
Then $\rt^n_k(V, W)$ admits strong cone avoidance for non-arithmetical cones.
\end{lemma}
\begin{proof}
Fix a set $Z$, a non-$Z$-arithmetical set $C$ and a coloring $f : [\omega]^n \to k$.

Suppose first that $C$ is not $Z$-hyperarithmetical. By Solovay~\cite{Solovay1978Hyperarithmetically}, $C$ is not computably encodable relative to~$Z$. Since for every infinite set $X \subseteq \omega$, there is an infinite $\arithscart^n_k$-solution $Y \subseteq X$ to $f$, there is an $\arithscart^n_k$-solution $H$ to $f$ such that $C \not \leq_T Z \oplus H$. 

Suppose now that $C$ is $Z$-hyperarithmetical. By Groszek and Slaman~\cite{Groszek2007Moduli}, there is a modulus $\mu : \omega \to \omega$ relative to $Z$, that is, for every function $g$ dominating $\mu$, $C \leq_T Z \oplus g$.
Let $f_1$ be defined for each $D \in [\omega]^n$ by $f_1(D) = \langle f(D), \Vcal_n(\mu, D) \rangle$.
By strong cone avoidance of $\rt^n_{<\infty, 2^{n-1}}$ for non-arithmetical cones (see Cholak and Patey~\cite[Theorem 4.15]{Cholak2019Thin}), there is an infinite set $H$ such that $C \not \leq_T Z \oplus H$ and $|f_1[H]^n| \leq 2^{n-1}$. In particular, $Z \oplus H$ does not compute a function dominating~$\mu$, so by Theorem~\ref{thm:avoid-colors-dominate-mu}, for every $\Gcal \in \Vcal_n$, there is some $i < k$ and some $D \in [H]^n$ such that $f_1(D) = \langle i, \Gcal \rangle$. Since $|\Vcal_n| = 2^{n-1}$, this $i$ is unique. For each $\Gcal \in \Vcal_n$, let $\chi(\Gcal)$ be such an~$i$.
	In particular, for every $D \in [H]^n$, $f(D) = \chi(\Vcal_n(\mu, D))$.
	By Cholak and Patey~\cite[Corollary 5.5]{Cholak2019Thin}, there is an infinite $\mu$-transitive subset $H_1 \subseteq H$ such that $C \not \leq_T Z \oplus H_1$. Therefore, $H_1$ is a $\arithscart^n_k$-solution to~$f$.
	
If $\chi \in \Rcal^n_k(V)$, then since $\rt^n_k(W) \leq_{id} \arithscartp{\Rcal^n_k(V)}^n_k$, $H_2$ is an $\rt^n_k(V, W)$-solution to $f$, and we are done. So suppose $\chi \not \in \Rcal^n_k(V)$. Unfolding the definition, $\rt^n_k(V) \not \leq_{id} \arithscartp{\{\chi\}}^n_k$, so there is a coloring $f_\chi : [\omega]^n \to k$ and some infinite $\arithscart^n_k$-solution $H_\chi$ to~$f_\chi$ with witnesses $\mu_\chi : \omega \to \omega^{+}$ and $\chi : \Vcal_n \to k$, that meets some $\rt^n_k$-pattern $P_\chi \in V$. 

However, $H_2$ $f$-avoids $P_\chi$ since $f$ is an instance of $\rt^n_k(V, W)$. By Lemma~\ref{lem:arithscart-avoid-pattern-dominates-mu}, $H_2 \oplus Z$ computes a function dominating~$\mu$, hence $C \leq_T Z \oplus H_2$, contradiction. This completes the proof of Lemma~\ref{lem:rt-v-w-sca-arith}.
\end{proof}

\begin{lemma}\label{lem:w-not-fixed-arithscart-but-v-does}
Let $V$ and $W$ be collections of $\rt^n_k$-patterns such that
$\rt^n_k(W) \not \leq_{id} \arithscartp{\Rcal^n_k(V)}^n_k$. 
For every function $\mu : \omega \to \omega$, there is an $\rt^n_k(V, W)$-instance such that every solution computes a function dominating~$\mu$.
\end{lemma}
\begin{proof}
Since $\rt^n_k(W) \not \leq_{id} \arithscartp{\Rcal^n_k(V)}^n_k$, there is some coloring $f_{fail} : [\omega]^n \to k$ and some infinite $\arithscart^n_k$-solution $H$ to~$f$ with some witness $\mu_{fail} : \omega \to \omega$ and $\chi \in \Rcal^n_k(V)$ that meets some $\rt^n_k$-pattern $P \in W$.
Define $f : [\omega]^n \to k$ by $f(D) = \chi(\Vcal_n(\mu, D))$.
Suppose $\mu$ is not dominated by any computable function, otherwise we are done.
By Cholak and Patey~\cite[Theorem 5.11]{Cholak2019Thin}, there is an infinite $\mu$-transitive set $H$ which does not compute a function dominating $\mu$.
Therefore, $H$ is an $\arithscart^n_k$-solution to $f$ with witnesses $\mu$ and $\chi \in \Rcal^n_k(V)$. By definition of $\Rcal^n_k(V)$, $\rt^n_k(V) \leq_{id} \arithscartp{\{\chi\}}^n_k$ and $H$ is an $\rt^n_k(V)$-solution to $f$, so $f : [H]^n \to k$ is an instance of $\rt^n_k(V, W)$.

We claim that for every $\rt^n_k(V, W)$-solution $G \subseteq H$ to $f$, $G$ computes a function dominating $\mu$. In particular, $G$ $f$-avoids $P$, so by Lemma~\ref{lem:arithscart-avoid-pattern-dominates-mu}, $G$ computes a function dominating~$\mu$. This completes the proof of Lemma~\ref{lem:w-not-fixed-arithscart-but-v-does}.
\end{proof}

\begin{theorem}\label{thm:characterization-sca-arith-promise-rt22-problems}
Let $V$ and $W$ be two collections of $\rt^n_k$-patterns.
$\rt^n_k(V, W)$ admits strong cone avoidance for non-arithmetical cones if and only if $\rt^n_k(W) \leq_{id} \arithscartp{\Rcal^n_k(V)}^n_k$.
\end{theorem}
\begin{proof}
$\Leftarrow$: This is Lemma~\ref{lem:rt-v-w-sca-arith}.  $\Rightarrow$: We prove the contrapositive. Suppose $\rt^n_k(W) \not \leq_{id} \arithscartp{\Rcal^n_k(V)}^n_k$ Let $C$ be a non-arithmetical set with modulus $\mu$. By Lemma~\ref{lem:w-not-fixed-arithscart-but-v-does}, there is a $\rt^n_k(V, W)$-instance such that every solution computes a function dominating $\mu$, hence computes~$C$.
\end{proof}

In the case $V = \emptyset$, then $\Rcal^n_k(V)$ is the set of all functions $\chi : \Vcal_n \to k$ and therefore $\arithscartp{\Rcal^n_k(V)}^n_k$ is $\arithscart^n_k$. We obtain Theorem~\ref{thm:arithscart-sca-characterization}, that is, $\rt^n_k(\emptyset, W)$ admits strong cone avoidance for non-arithmetical cones if and only if $\rt^n_k(W) \leq_{id} \arithscart^n_k$.
More interestingly, we consider the case where $W$ is the set $W_{\rt^n_k}$ of patterns forbidding non-homogeneous sets.

\begin{corollary}\label{cor:rt-v-rtnk-sca-arith}
Let $V$ be a collection of $\rt^n_k$-patterns.
$\rt^n_k(V, W_{\rt^n_k})$ admits strong cone avoidance for non-arithmetical cones if and only if $\Rcal^n_k(V)$ contains only constant functions.
\end{corollary}
\begin{proof}
By Theorem~\ref{thm:characterization-sca-arith-promise-rt22-problems}, 
$\rt^n_k(V, W_{\rt^n_k})$ admits strong cone avoidance for non-arithmetical cones if and only if $\rt^n_k \leq_{id} \arithscartp{\Rcal^n_k(V)}^n_k$.
Case 1: $\Rcal^n_k(V)$ contains only constant functions. Then for any $f : [\omega]^n \to k$, any $\arithscartp{\Rcal^n_k(V)}^n_k$-solution to $f$ is $f$-homogeneous, hence $\rt^n_k \leq_{id} \arithscartp{\Rcal^n_k(V)}^n_k$, so $\rt^n_k(V, W_{\rt^n_k})$ admits strong cone avoidance for non-arithmetical cones.
Case 2: $\Rcal^n_k(V)$ contains a non-constant function $\chi : \Vcal_n \to k$.
Let $\mu : \omega \to \omega$ be a modulus of a non-arithmetical set $C$.
Let $f(D) = \chi(\Vcal_n(\mu, D))$. By Cholak and Patey~\cite[Corollary 5.5]{Cholak2019Thin}, there is an infinite $\mu$-transitive set $H$ which does not compute a function dominating $\mu$. Therefore, $H$ is an $\arithscartp{\Rcal^n_k(V)}^n_k$-solution to $f$ with witnesses $\mu$ and $\chi$. We claim that $H$ is not $f$-homogeneous. Indeed, by Theorem~\ref{thm:avoid-colors-dominate-mu}, for every $\Gcal \in \Vcal_n$, there is some $D \in [H]^n$ such that $\Gcal = \Vcal_n(\mu, D)$. Since $\chi$ is not constant on $\Vcal_n$, $H$ is not $f$-homogeneous. Thus $\rt^n_k \not\leq_{id} \arithscartp{\Rcal^n_k(V)}^n_k$ and $\rt^n_k(V, W_{\rt^n_k})$ does not admit strong cone avoidance for non-arithmetical cones.
\end{proof}

\subsection{Strongly avoiding non-computable cones} The strong cone avoidance analysis is very similar to the one for non-arithmetical cones, \emph{mutatis mutandis}. We start again by defining a restricted version of $\largeo^n_k$.

\begin{statement}
Let $\Scal$ be a set of functions of type $\chi : \Lcal_n \to k$.
$\largep{\Scal}^n_k$: For every function $f : [\omega]^n \to k$,
there is a strongly increasing relative left-c.e.\ function $\mu : \omega \to \omega^{+}$, an infinite strongly $\mu$-transitive set $H \subseteq \omega$, and a coloring $\chi \in \Scal$ such that for every $D \in [H]^n$, $f(D) = \chi(\Lcal_n(\mu, D))$.
\end{statement}

Given a collection of $\rt^n_k$-patterns, we define a class $\Scal^n_k(W)$ for $\largeo^n_k$ is a similar way as the class $\Rcal^n_k(W)$ for $\arithscart^n_k$, that is, given some $n, k \in \omega$ and a collection $W$ of $\rt^n_k$-patterns,
let 
$$
\Scal^n_k(W) = \left\{\chi : \Lcal_n \to k \mbox{ }|\mbox{ } \rt^n_k(W) \leq_{id} \largep{\{\chi\}}^n_k\right\}
$$
In particular, if $\rt^n_k(W)$ is a true statement, then every constant function $\chi : \Lcal_n \to k$ belongs to~$\Scal^n_k(W)$.

\begin{lemma}\label{lem:rt-v-w-sca}
Let $V$ and $W$ be collections of $\rt^n_k$-patterns such that
$\rt^n_k(W) \leq_{id} \largep{\Scal^n_k(V)}^n_k$. 
Then $\rt^n_k(V, W)$ admits strong cone avoidance.
\end{lemma}
\begin{proof}
Fix two sets $Z$ and $C$ with $C \not \leq_T Z$ and let $f : [\omega]^n \to k$ be an instance of $\largeo^n_k$.

By Lerman~\cite[4.18]{Lerman1983Degrees}, there is a set $Z_1 \geq_T Z$ such that $C$ is $\Delta^0_2(Z_1)$ but $C \not \leq_T Z_1$. Since $C$ is $\Delta^0_2(Z_1)$, there is a left $Z_1$-c.e.\ modulus $\mu : \omega \to \omega$ for $C$. By Lemma~\ref{lem:leftce-computes-strongly-increasing}, we can assume that $\mu$ is strongly increasing.

Let $f : [\omega]^n \to k$ be an instance of $\rt^n_k(V, W)$.
Let $f_1 : [\omega]^n \to k \times \Lcal_n$ be defined for each $D \in [\omega]^n$ by $f_1(D) = \langle f(D), \Lcal_n(\mu, D) \rangle$.
By strong cone avoidance of $\rt^n_{<\infty, C_n}$ (Cholak and Patey~\cite[Theorem 4.18]{Cholak2019Thin}), there is an infinite set $H \subseteq \omega$ such that $C \not \leq_T Z_1 \oplus H$ and $|f_1[H]^n| \leq C_n$. In particular, $Z_1 \oplus H$ does not compute a function dominating $\mu$, so by Theorem~\ref{thm:avoid-colors-dominate-left-ce-mu}, for every $\Gcal \in \Lcal_n$, there is some $i < k$ and some $D \in [H]^n$ such that $f_1(D) = \langle i, \Gcal \rangle$. Since $|\Lcal_n| = C_n$ (Cholak and Patey~\cite[Lemma 3.16]{Cholak2019Thin}), this $i$ is unique. For each $\Gcal$, let $\chi(\Gcal)$ be such an~$i$. Then for every $D \in [H]^n$, $f(D) = \chi(\Lcal_n(\mu, D))$. By Cholak and Patey~\cite[Corollary 5.5]{Cholak2019Thin}, there is an infinite $\mu$-transitive subset $H_1 \subseteq H$ such that $C \not \leq_T Z \oplus H_1$. By Lemma~\ref{lem:transitive-to-strongly-transitive}, there is an infinite strongly $\mu$-transitive subset $H_2 \subseteq H_1$ such that $C \not \leq_T Z \oplus H_2$. Therefore, $H_2$ is a $\largeo^n_k$-solution to~$f$.

If $\chi \in \Scal^n_k(V)$, then since $\rt^n_k(W) \leq_{id} \largep{\Scal^n_k(V)}^n_k$, $H_2$ is an $\rt^n_k(V, W)$-solution to $f$, and we are done. So suppose $\chi \not \in \Scal^n_k(V)$. Unfolding the definition, $\rt^n_k(V) \not \leq_{id} \largep{\{\chi\}}^n_k$, so there is a coloring $f_\chi : [\omega]^n \to k$ and some infinite $\largeo^n_k$-solution $H_\chi$ to~$f_\chi$ with witnesses $\mu_\chi : \omega \to \omega^{+}$ and $\chi : \Lcal_n \to k$, that meets some $\rt^n_k$-pattern $P_\chi \in V$. 

However, $H_2$ $f$-avoids $P_\chi$ since $f$ is an instance of $\rt^n_k(V, W)$. By Lemma~\ref{lem:largeo-avoid-pattern-dominates-mu}, $H_2 \oplus Z$ computes a function dominating~$\mu$, hence $C \leq_T Z \oplus H_2$, contradiction. This completes the proof of Lemma~\ref{lem:rt-v-w-sca}.
\end{proof}

\begin{lemma}\label{lem:w-not-fixed-largeo-but-v-does}
Let $V$ and $W$ be collections of $\rt^n_k$-patterns such that
$\rt^n_k(W) \not \leq_{id} \largep{\Scal^n_k(V)}^n_k$. 
For every strongly increasing left-c.e.\ function $\mu : \omega \to \omega$, there is an $\rt^n_k(V, W)$-instance such that every solution computes a function dominating~$\mu$.
\end{lemma}
\begin{proof}
Since $\rt^n_k(W) \not \leq_{id} \largep{\Scal^n_k(V)}^n_k$, there is some coloring $f_{fail} : [\omega]^n \to k$ and some infinite $\largeo^n_k$-solution $H_{fail}$ to~$f_{fail}$ with some witness $\mu_{fail} : \omega \to \omega$ and $\chi \in \Scal^n_k(V)$ that meets some $\rt^n_k$-pattern $P \in W$.
Define $f : [\omega]^n \to k$ by $f(D) = \chi(\Lcal_n(\mu, D))$.
Suppose that $\mu$ is not dominated by a computable function.
By Cholak and Patey~\cite[Theorem 5.11]{Cholak2019Thin} and Lemma~\ref{lem:transitive-to-strongly-transitive}, there is an infinite strongly $\mu$-transitive set $H$ which does not compute a function dominating $\mu$. Therefore, $H$ is a $\largeo^n_k$-solution to $f$ with witnesses $\mu$ and $\chi \in \Scal^n_k(V)$. By definition of $\Scal^n_k(V)$, $\rt^n_k(V) \leq_{id} \largep{\{\chi\}}^n_k$ and $H$ is an $\rt^n_k(V)$-solution to $f$, so $f : [H]^n \to k$ is an instance of $\rt^n_k(V, W)$.

We claim that for every $\rt^n_k(V, W)$-solution $G \subseteq H$ to $f$, $G$ computes a function dominating $\mu$. In particular, $G$ $f$-avoids $P$, so by Lemma~\ref{lem:largeo-avoid-pattern-dominates-mu}, $G$ computes a function dominating~$\mu$. This completes the proof of Lemma~\ref{lem:w-not-fixed-largeo-but-v-does}.
\end{proof}

\begin{theorem}\label{thm:characterization-sca-promise-rt22-problems}
Let $V$ and $W$ be two collections of $\rt^n_k$-patterns.
$\rt^n_k(V, W)$ admits strong cone avoidance if and only if $\rt^n_k(W) \leq_{id} \largep{\Scal^n_k(V)}^n_k$.
\end{theorem}
\begin{proof}
$\Leftarrow$: This is Lemma~\ref{lem:rt-v-w-sca}. $\Rightarrow$: We prove the contrapositive. Suppose $\rt^n_k(W) \not \leq_{id} \largep{\Scal^n_k(V)}^n_k$. Let $\mu$ be a strongly increasing left-c.e. modulus of $\emptyset'$. By Lemma~\ref{lem:w-not-fixed-largeo-but-v-does}, there is a $\rt^n_k(V, W)$-instance such that every solution computes a function dominating $\mu$, hence computes~$\emptyset'$.
\end{proof}

\begin{corollary}\label{cor:rt-v-rtnk-sca}
Let $V$ be a collection of $\rt^n_k$-patterns.
$\rt^n_k(V, W_{\rt^n_k})$ admits strong cone avoidance if and only if $\Scal^n_k(V)$ contains only constant functions.
\end{corollary}
\begin{proof}
By Theorem~\ref{thm:characterization-sca-promise-rt22-problems}, 
$\rt^n_k(V, W_{\rt^n_k})$ admits strong cone avoidance if and only if $\rt^n_k \leq_{id} \largep{\Scal^n_k(V)}^n_k$.
Case 1: $\Scal^n_k(V)$ contains only constant functions. Then for any $f : [\omega]^n \to k$, any $\largep{\Scal^n_k(V)}^n_k$-solution to $f$ is $f$-homogeneous, hence $\rt^n_k \leq_{id} \largep{\Scal^n_k(V)}^n_k$, so $\rt^n_k(V, W_{\rt^n_k})$ admits strong cone avoidance.
Case 2: $\Scal^n_k(V)$ contains a non-constant function $\chi : \Lcal_n \to k$.
Let $\mu : \omega \to \omega$ be a strongly increasing left-c.e\ modulus of a non-computable set $C$.
Let $f(D) = \chi(\Lcal_n(\mu, D))$. By Cholak and Patey~\cite[Corollary 5.5]{Cholak2019Thin} and Lemma~\ref{lem:transitive-to-strongly-transitive}, there is an infinite strongly $\mu$-transitive set $H$ such that $C \not \leq_T H$.  Therefore, $H$ is an $\largep{\Scal^n_k(V)}^n_k$-solution to $f$ with witnesses $\mu$ and $\chi$. We claim that $H$ is not $f$-homogeneous. Since $C \not \leq_T H$, then in particular, $H$ does not compute a function dominating $\mu$, so by Theorem~\ref{thm:avoid-colors-dominate-left-ce-mu}, for every $\Gcal \in \Lcal_n$, there is some $D \in [H]^n$ such that $\Gcal = \Lcal_n(\mu, D)$. Since $\chi$ is not constant on $\Lcal_n$, $H$ is not $f$-homogeneous. Thus $\rt^n_k \not \leq_{id} \largep{\Scal^n_k(V)}^n_k$ and $\rt^n_k(V, W_{\rt^n_k})$ does not admit strong cone avoidance.
\end{proof}

\subsection{Avoiding non-computable cones} Last, we complete our analysis of this extended class of promise Ramsey-like theorems with cone avoidance. The analysis follows the same scheme.

\begin{statement}
Let $\Tcal$ be a set of functions of type $\chi : \Pcal_n \to k$.
$\lcelargep{\Tcal}^n_k$: For every function $f : [\omega]^n \to k$,
there is a strongly increasing relative left-c.e.\ function $\mu : \omega \to \omega^{+}$, an infinite strongly $\mu$-transitive set $H \subseteq \omega$, and a coloring $\chi \in \Tcal$ such that for every $D \in [H]^n$, $f(D) = \chi(\Pcal_n(\mu, D))$.
\end{statement}

We now define $\Tcal^n_k(W)$ the same way we defined $\Rcal^n_k(W)$ and $\Scal^n_k(W)$ for $\arithscart^n_k$ and $\largeo^n_k$, respectively.
Given some $n, k \in \omega$ and a collection $W$ of $\rt^n_k$-patterns,
let 
$$
\Tcal^n_k(W) = \left\{\chi : \Pcal_n \to k \mbox{ }|\mbox{ } \rt^n_k(W) \leq_{id} \lcelargep{\{\chi\}}^n_k\right\}
$$
In particular, if $\rt^n_k(W)$ is a true statement, then every constant function $\chi : \Pcal_n \to k$ belongs to~$\Tcal^n_k(W)$.

\begin{lemma}\label{lem:rt-v-w-ca}
Let $V$ and $W$ be collections of $\rt^n_k$-patterns such that $\rt^n_k(W) \leq_{id} \lcelargep{\Tcal^n_k(V)}^n_k$.
Then $\rt^n_k(V, W)$ admits cone avoidance.
\end{lemma}
\begin{proof}
Fix two sets $Z$ and $C$ with $C \not \leq_T Z$ and let $f : [\omega]^n \to k$ be a $Z$-computable instance of $\lcelargeo^n_k$.

By Lerman~\cite[4.18]{Lerman1983Degrees}, there is a set $Z_1 \geq_T Z$ such that $C$ is $\Delta^0_2(Z_1)$ but $C \not \leq_T Z_1$. Since $C$ is $\Delta^0_2(Z_1)$, there is a left $Z_1$-c.e.\ modulus $\mu : \omega \to \omega$ for $C$. By Lemma~\ref{lem:leftce-computes-strongly-increasing}, we can assume that $\mu$ is strongly increasing.

Let $f_1 : [\omega]^n \to k \times \Pcal_n$ be defined for each $D \in [\omega]^n$ by $f_1(D) = \langle f(D), \Pcal_n(\mu, D) \rangle$.
By cone avoidance of $\rt^n_{<\infty, C_{n-1}}$ (Cholak and Patey~\cite[Corollary 4.19]{Cholak2019Thin}), there is an infinite set $H \subseteq \omega$ such that $C \not \leq_T Z_1 \oplus H$ and $|f_1[H]^n| \leq C_{n-1}$. In particular, $Z_1 \oplus H$ does not compute a function dominating~$\mu$, so by Theorem~\ref{thm:avoid-colors-dominate-left-ce-mu}, for every $\Gcal \in \Pcal_n$, there is some $i < k$ and some $D \in [H]^n$ such that $f_1(D) = \langle i, \Gcal \rangle$. Since $|\Pcal_n| = C_{n-1}$ (Cholak and Patey~\cite[Lemma 3.15,Lemma 3.16]{Cholak2019Thin}), this $i$ is unique. For each $\Gcal$, let $\chi(\Gcal)$ be such an~$i$. 
For every $D \in [H]^n$, $f(D) = \chi(\Pcal_n(\mu, D))$. By Cholak and Patey~\cite[Corollary 5.5]{Cholak2019Thin}, there is an infinite $\mu$-transitive subset $H_1 \subseteq H$ such that $C \not \leq_T Z \oplus H_1$. By Lemma~\ref{lem:transitive-to-strongly-transitive}, there is an infinite strongly $\mu$-transitive subset $H_2 \subseteq H_1$ such that $C \not \leq_T Z \oplus H_2$. Therefore, $H_2$ is a $\lcelargeo^n_k$-solution to~$f$.

If $\chi \in \Tcal^n_k(V)$, then since $\rt^n_k(W) \leq_{id} \lcelargep{\Tcal^n_k(V)}^n_k$, $H_2$ is an $\rt^n_k(V, W)$-solution to $f$, and we are done. So suppose $\chi \not \in \Tcal^n_k(V)$. Unfolding the definition, $\rt^n_k(V) \not \leq_{id} \lcelargep{\{\chi\}}^n_k$, so there is a coloring $f_\chi : [\omega]^n \to k$ and some infinite $\lcelargeo^n_k$-solution $H_\chi$ to~$f_\chi$ with witnesses $\mu_\chi : \omega \to \omega^{+}$ and $\chi : \Pcal_n \to k$, that meets some $\rt^n_k$-pattern $P_\chi \in V$. 

However, $H_2$ $f$-avoids $P_\chi$ since $f$ is an instance of $\rt^n_k(V, W)$. By Lemma~\ref{lem:lcelargeo-avoid-pattern-dominates-mu}, $H_2 \oplus Z$ computes a function dominating~$\mu$, hence $C \leq_T Z \oplus H_2$, contradiction. This completes the proof of Lemma~\ref{lem:rt-v-w-ca}.
\end{proof}

\begin{lemma}\label{lem:w-not-fixed-lcelargeo-but-v-does}
Let $V$ and $W$ be collections of $\rt^n_k$-patterns such that
$\rt^n_k(W) \not \leq_{id} \lcelargep{\Tcal^n_k(V)}^n_k$. 
For every strongly increasing left-c.e.\ function $\mu : \omega \to \omega$, there is a set $Z$ which does not compute a function dominating $\mu$, and a $Z$-computable $\rt^n_k(V, W)$-instance such that every solution $Z$-computes a function dominating~$\mu$.
\end{lemma}
\begin{proof}
Since $\rt^n_k(W) \not \leq_{id} \lcelargep{\Tcal^n_k(V)}^n_k$, there is some coloring $f_{fail} : [\omega]^n \to k$ and some infinite $\lcelargeo^n_k$-solution $H$ to~$f$ with some witness $\mu_{fail} : \omega \to \omega^{+}$ and $\chi \in \Tcal^n_k(V)$ that meets some $\rt^n_k$-pattern $P \in W$.
Define $f : [\omega]^n \to k$ by $f(D) = \chi(\Pcal_n(\mu, D))$. Suppose that $\mu$ is not dominated by a computable function.
By Cholak and Patey~\cite[Theorem 5.11]{Cholak2019Thin} and Lemma~\ref{lem:transitive-to-strongly-transitive}, there is an infinite strongly $\mu$-transitive set $H$ which does not compute a function dominating $\mu$. Therefore, $H$ is a $\lcelargeo^n_k$-solution to $f$ with witnesses $\mu$ and $\chi \in \Tcal^n_k(V)$. By definition of $\Tcal^n_k(V)$, $\rt^n_k(V) \leq_{id} \lcelargep{\{\chi\}}^n_k$ and $H$ is an $\rt^n_k(V)$-solution to $f$, so $f : [H]^n \to k$ is an $H$-computable instance of $\rt^n_k(V, W)$ such that $H$ does not compute a function dominating $\mu$.

We claim that for every $\rt^n_k(V, W)$-solution $G \subseteq H$ to $f$, $G \oplus H$ computes a function dominating $\mu$. In particular, $G$ $f$-avoids $P$, so by Lemma~\ref{lem:lcelargeo-avoid-pattern-dominates-mu}, $G \oplus H$ computes a function dominating~$\mu$. This completes the proof of Lemma~\ref{lem:w-not-fixed-lcelargeo-but-v-does}.
\end{proof}

\begin{theorem}\label{thm:characterization-ca-promise-rt22-problems}
Let $V$ and $W$ be two collections of $\rt^n_k$-patterns.
$\rt^n_k(V, W)$ admits cone avoidance if and only if $\rt^n_k(W) \leq_{id} \lcelargep{\Tcal^n_k(V)}^n_k$.
\end{theorem}
\begin{proof}
$\Leftarrow$: This is Lemma~\ref{lem:rt-v-w-ca}. $\Rightarrow$: We prove the contrapositive. Suppose $\rt^n_k(W) \not \leq_{id} \lcelargep{\Tcal^n_k(V)}^n_k$. Let $\mu$ be a strongly increasing left-c.e. modulus of $\emptyset'$. By Lemma~\ref{lem:w-not-fixed-lcelargeo-but-v-does}, there is a set $Z$ which does not compute a function dominating $\mu$ and a $Z$-computable $\rt^n_k(V, W)$-instance such that every solution $Z$~computes a function dominating $\mu$, hence computes~$\emptyset'$. Therefore $\rt^n_k(W)$ does not admit cone avoidance.
\end{proof}

\begin{corollary}\label{cor:rt-v-rtnk-ca}
Let $V$ be a collection of $\rt^n_k$-patterns.
$\rt^n_k(V, W_{\rt^n_k})$ admits cone avoidance if and only if $\Tcal^n_k(V)$ contains only constant functions.
\end{corollary}
\begin{proof}
By Theorem~\ref{thm:characterization-ca-promise-rt22-problems}, 
$\rt^n_k(V, W_{\rt^n_k})$ admits cone avoidance if and only if $\rt^n_k \leq_{id} \lcelargep{\Tcal^n_k(V)}^n_k$.
Case 1: $\Tcal^n_k(V)$ contains only constant functions. Then for any $f : [\omega]^n \to k$, any $\lcelargep{\Tcal^n_k(V)}^n_k$-solution to $f$ is $f$-homogeneous, hence $\rt^n_k \leq_{id} \lcelargep{\Tcal^n_k(V)}^n_k$, so $\rt^n_k(V, W_{\rt^n_k})$ admits cone avoidance.
Case 2: $\Tcal^n_k(V)$ contains a non-constant function $\chi : \Pcal_n \to k$.
Let $\mu : \omega \to \omega$ be a strongly increasing left-c.e\ modulus of a non-computable set $C$.
Let $f(D) = \chi(\Pcal_n(\mu, D))$. By Cholak and Patey~\cite[Corollary 5.5]{Cholak2019Thin} and Lemma~\ref{lem:transitive-to-strongly-transitive}, there is an infinite strongly $\mu$-transitive set $H$ which does not compute $C$. Therefore, $H$ is a $\lcelargep{\Tcal^n_k(V)}^n_k$-solution to $f$ with witnesses $\mu$ and $\chi$. We claim that $H$ is not $f$-homogeneous. Indeed, since $C \not \leq_T H$, then $H$ does not compute a function dominating $\mu$, so by Theorem~\ref{thm:avoid-colors-dominate-left-ce-mu-computable}, for every $\Gcal \in \Pcal_n$, there is some $D \in [H]^n$ such that $\Gcal = \Pcal_n(\mu, D)$. Since $\chi$ is not constant on $\Pcal_n$, $H$ is not $f$-homogeneous. Thus $\rt^n_k \not \leq_{id} \lcelargep{\Tcal^n_k(V)}^n_k$ and $\rt^n_k(V, W_{\rt^n_k})$ does not admit cone avoidance.
\end{proof}

\section{Applications}\label{sect:applications}

In this section, we exemplify the use of this framework by reproving existing theorems in reverse mathematics without involving any forcing argument.

\subsection{Ramsey's theorem}

Ramsey's theorem for pairs is one of the most famous theorems studied in reverse mathematics as it is historically the first one which is not equivalent (and not even linearly ordered with) the five systems of axioms known as the Big Five~\cite{Simpson2009Subsystems}.
The first theorem that we reprove with our framework is the celebrated Seetapun theorem~\cite[Theorem 2.1]{Seetapun1995strength}, answering negatively the long-standing question whether $\rt^2_2$ is equivalent to the Arithmetical Comprehension Axiom ($\aca$) over $\rca$.

\begin{theorem}[Seetapun and Slaman~\cite{Seetapun1995strength}]
For every $k \geq 1$, $\rt^2_k$ admits cone avoidance.
\end{theorem}
\begin{proof}
By Theorem~\ref{thm:lcelargeo-ca}, $\lcelargeo^2_k$ admits cone avoidance.
As explained, $\lcelargeo^2_k$ is nothing but the statement $\rt^2_k$. Indeed, since $|\Pcal_2| = 1$, the function $\chi : \Pcal_2 \to k$ is constant, and $\lcelargeo^2_k$ asserts the existence of an infinite set $H$ over which $f$ belongs to the range of $\chi$. In particular $H$ is $f$-homogeneous.
\end{proof}

Later, Jockusch and Dzhafarov~\cite[Lemma 3.2]{Jockusch1972Ramseys} adapted the proof of Seetapun's theorem to obtain strong cone avoidance of $\rt^1_k$.

\begin{theorem}[Jockusch and Dzhafarov~\cite{Jockusch1972Ramseys}]
For every $k \geq 1$, $\rt^1_k$ admits strong cone avoidance.
\end{theorem}
\begin{proof}
By Theorem~\ref{thm:largeo-sca}, $\largeo^1_k$ admits strong cone avoidance.
Here again, $\largeo^1_k$ is nothing but the statement $\rt^1_k$, since $|\Lcal_1| = 1$.
\end{proof}

Wang~\cite[Theorem 3.1]{Wang2014Some} surprisingly proved that for every $n$, whenever $\ell$ is sufficiently large, then $\rt^n_{<\infty, \ell}$ admits strong cone avoidance. Cholak and Patey~\cite[Theorem 4.18]{Cholak2019Thin} improved his bound to Catalan's sequence, and proved the tightness of the result. The following theorem was used all over the article, and we can actually get a reversal.

\begin{theorem}[Cholak and Patey~\cite{Cholak2019Thin}]
For every $n \geq 1$, $\rt^n_{<\infty, C_n}$ admits strong cone avoidance.	
\end{theorem}
\begin{proof}
By Theorem~\ref{thm:largeo-sca}, $\largeo^n_k$ admits strong cone avoidance. 
In particular, $\largeo^n_k$ asserts the existence, for every coloring $f : [\omega]^n \to k$, of an infinite set $H$ and a function $\chi : \Lcal_n \to k$ such that $f[H]^n \subseteq \chi(\Lcal_n)$. Since $|\Lcal_n| = C_n$, it follows that $|f[H]^n| \leq C_n$ and therefore that $H$ is an $\rt^n_{<\infty, C_n}$-solution to~$f$.
\end{proof}

\subsection{The Erd\H{o}s-Moser theorem and the Ascending Descending sequence}

A \emph{tournament} $T$ is a directed graph such that any two vertices has exactly one arrow. We consider the predicate $T(x, y)$ to be true if there is an arrow from $x$ to $y$. A set $H$ is \emph{$T$-transitive} if for every $x, y, z \in H$ such that $T(x, y)$ and $T(y, z)$ both hold, then $T(x, z)$ holds. 

\begin{statement}[Erd\H{o}s-Moser]
$\emo$: Every infinite tournament admits an infinite transitive set.
\end{statement}

A tournament $T$ can be represented by a coloring $f : [\omega]^2 \to 2$ such that for every $x < y$, $f(x, y) = 1$ if and only if $T(x, y)$ holds. A set $H$ is $f$-transitive if for every $x < y < z \in H$ and $i < 2$, if $f(x, y) = f(y, z)= i$ then $f(x, z) = i$. One can see the Erd\H{o}s-Moser theorem as the Ramsey-like statement \qt{For every coloring $f : [\omega]^2 \to 2$, there is an infinite $f$-transitive set $H$.}

The following $\ads$ principle was introduced and studied by Hirschfeldt and Shore~\cite{Hirschfeldt2007Combinatorial} in the context of reverse mathematics. 

\begin{statement}[Ascending Descending sequence]
$\ads$: Every infinite linear order has an infinite ascending or descending sequence.
\end{statement}

As explained, a linear order $(\omega, \prec_\Lcal)$ can be represented by a coloring $f : [\omega]^2 \to 2$ such that for every $x < y \in \omega$, $f(x, y) = 1$ if and only if $x \prec_\Lcal y$. In particular, $\omega$ is $f$-transitive.
The Ascending Descending sequence principle can be seen as the promise Ramsey-like statement \qt{For every coloring $f : [\omega]^2 \to 2$ such that $\omega$ is $f$-transitive, there is an infinite $f$-homogeneous set.}

The $\emo$ principle was introduced by Bovykin and Weiermann~\cite{Bovykin2005strength} as a way to decompose the proof of $\rt^2_2$ into two steps. Indeed, given a coloring $f : [\omega]^2 \to 2$, by $\emo$, there is an infinite set $X$ over which $f$ is transitive, and by $\ads$, there is an infinite $f$-homogeneous subset $Y \subseteq X$. The author~\cite{PateyCombinatorial} proved that $\emo$ admits strong cone avoidance.

\begin{theorem}[Patey~\cite{PateyCombinatorial}]
$\emo$ admits strong cone avoidance.
\end{theorem}
\begin{proof}
By Theorem~\ref{thm:largeo-sca-characterization}, it suffices to prove that $\emo \leq_{id} \largeo^2_2$. However, $\largeo^2_2$ is the statement $\explicitlarge_2$ saying that for every coloring $f : [\omega]^2 \to 2$, there is some $i < 2$ and an infinite set $H$ such that for every $x < y < z \in H$, $f(x, y) = f(y, z) = i$ if and only if $f(x, z) = i$. In particular $H$ is $f$-transitive.
\end{proof}

The author~\cite{PateyCombinatorial} also deduced that $\ads$ does not admit strong cone avoidance. Indeed, if $\ads$ and $\emo$ both admit strong cone avoidance, then $\rt^2_2$ does, which is known not to be the case. One can however reprove it directly from our criterion for promise Ramsey-like theorems.

\begin{theorem}[Patey~\cite{PateyCombinatorial}]
$\ads$ does not admit strong cone avoidance.
\end{theorem}
\begin{proof}
Let $V_{\ads}$ be the set of $\rt^2_2$-patterns forbidding non-transitive colorings. $\ads$ is the promise Ramsey-like statement $\rt^2_2(V_{\ads}, W_{\rt^2_2})$. By Corollary~\ref{cor:rt-v-rtnk-sca}, $\ads$ admits strong cone avoidance if and only if $\Scal^2_2(V_{\ads})$ contains only constant functions. Since $\rt^2_2(V_{\ads}) \leq_{id} \explicitlarge_2$ and $\explicitlarge_2$ is the statement $\largeo^2_2$, then $\Scal^2_2(V_{\ads})$ contains all the functions $\chi : \Lcal_2 \to 2$. Since $|\Lcal_2| = 2$, $\ads$ does not admit strong cone avoidance.
\end{proof}

Dorais et al~\cite{Dorais2016uniform} studied combinatorial principles, including Ramsey's theorem, under the Weihrauch reduction. They introduced for this some new consequences of Ramsey's theorem for pairs, such as $\sher$. Given a coloring $f : [\omega]^2 \to k$, a set $H \subseteq \omega$ is \emph{$f$-semi-hereditary} if for all $i < k$ except possibly one, whenever $x < y < z \in H$ and $f(x, z) = f(y, z) = i$, then $f(x, y) = i$.

\begin{statement}[Semi-hereditary]
$\sher_k$: Every coloring $f : [\omega]^2 \to k$ such that $\omega$ is $f$-semi-hereditary has an infinite $f$-homogeneous set.
\end{statement}

The restriction $\sher_2$ was studied by Dorais (unpublished), who showed that it follows from $\ads$. Thanks to our general criterion for promise Ramsey-like principles, we can prove the following theorem.

\begin{theorem}
For every $k \geq 2$, $\sher_k$ does not admit strong cone avoidance.
\end{theorem}
\begin{proof}
Let $V_{\sher_k}$ be the set of $\rt^2_k$-patterns forbidding non-semi-hereditary colorings. $\sher_k$ is the promise Ramsey-like statement $\rt^2_2(V_{\sher_k}, W_{\rt^2_k})$. By Corollary~\ref{cor:rt-v-rtnk-sca}, $\sher_k$ admits strong cone avoidance if and only if $\Scal^2_k(V_{\sher_k})$ contains only constant functions. 

We now prove that $\sher_k \leq_{id} \explicitlarge_k$.
Fix a coloring $f : [\omega]^2 \to k$ and let $H$ be an infinite $\explicitlarge_k$-solution to~$f$. In particular, there are two colors $i_s, i_\ell < k$ such that $f[H]^2 \subseteq \{i_s, i_\ell\}$ and for every $x < y < z \in H$, $f(x, y) = f(y, z) = i_s$ if and only if $f(x, z) = i_s$. Therefore, for all colors $i < k$ except possibly $i_\ell$, whenever $f(x, z) = f(y, z) = i$, then $f(x, y) = i$.

Since $\explicitlarge_k$ is the statement $\largeo^2_k$, then $\Scal^2_k(V_{\sher_k})$ contains all the functions $\chi : \Lcal_2 \to k$. Since $|\Lcal_2| = 2$, $\sher_k$ does not admit strong cone avoidance
\end{proof}

\subsection{The free set theorem} We now provide an example of application to a statement which involves $\omega$-colorings of $[\omega]^n$. The free set theorem was introduced by Friedman~\cite{FriedmanFom53free} and then studied by Cholak et al~\cite{Cholak2001Free} in the framework of reverse mathematics. Given a coloring $f : [\omega]^n \to \omega$, a set $H \subseteq \omega$ is \emph{$f$-free} if for every $D \in [H]^n$, whenever $f(D) \in H$ then $f(D) \in D$. Equivalently, $H$ is $f$-free if for every $x \in H$, $x \not \in f[H \setminus \{x\}]^n$.

\begin{statement}[Free set theorem]
$\fs^n$: Every coloring $f : [\omega]^n \to \omega$ has an infinite $f$-free set.
\end{statement}

Given a coloring $f : [\omega]^n \to \omega$, for every $D = \{x_0 < \dots < x_n \} \in [\omega]^{n+1}$, let $g(D)$ be the function $g_D : [n+1]^n \to 2$ defined for every $E \in [n+1]^n$ by $g_D(E) = 1$ if and only if $f(\{x_i : i \in E\}) = x_j$ where $j$ is the unique value in $D \setminus E$. Since there are $2^{n+1}$ functions of type $[n+1]^n \to 2$, one can see $g$ as a function of type $[\omega]^{n+1} \to 2^{n+1}$.
Moreover, every $g$-homogeneous set $H$ must be for color the 0-constant function of type $[n+1]^n \to 2$, and $H$ is then $f$-free.
We say that an infinite set $H$ is \emph{$g$-functional} if for every $D_1 = \{x_0 < \dots < x_n\}, D_2 = \{y_0 < \dots < y_n \} \in [\omega]^{n+1}$ and $E \in [n+1]^n$, if $g(D_1)(E) = g(D_2)(E) = 1$, then $\{x_i : i \in E\} = \{y_i : i \in E\}$. In particular, $\omega$ is $g$-functional for the coloring $g : [\omega]^{n+1} \to 2^{n+1}$ defined above. 
We can therefore see $\fs^n$ as the promise Ramsey-like statement \qt{For every coloring $g : [\omega]^{n+1} \to 2^{n+1}$ such that $\omega$ is $g$-functional, there is an infinite $g$-homogeneous set.}
Wang~\cite[Theorem 4.1]{Wang2014Some} proved the following theorem.

\begin{theorem}[Wang~\cite{Wang2014Some}]
For every $n \geq 1$, $\fs^n$ admits strong cone avoidance.
\end{theorem}
\begin{proof}
Let $V_{\fs^n}$ be the set of $\rt^{n+1}_{2^{n+1}}$-patterns which ensures that $\omega$ is $g$-functional. By Corollary~\ref{cor:rt-v-rtnk-sca}, $\fs^n$ admits strong cone avoidance if and only if $\Scal^{n+1}_{2^{n+1}}(V_{\fs^n})$ contains only constant functions. 

Let $\chi : \Lcal_{n+1} \to 2^{n+1}$ be a non-constant function. In particular, there is some $\Gcal \in \Lcal_{n+1}$ such that $\chi(\Gcal) : [n+1]^n \to 2$ is different from the 0-constant function of type $[n+1]^n \to 2$.

Let $\mu : \omega \to \omega$ be a strongly increasing left-c.e.\ modulus of~$\emptyset'$. Let $\rho : \omega \to \omega$ be the left-c.e.\ function defined for every $n, x \in \omega$ by $\rho_{2n}(2x) = 2\mu_n(x)$, $\rho_{2n}(2x+1) = \rho_{2n}(2x)$ and $\rho_{2n+1} = \rho_{2n}$. Informally, $\rho$ is obtained from $\mu$ by considering the integers of $\mu$ as even numbers for $\rho$, and interleaving odd numbers which do not change the value of $\rho$. We claim that $\rho$ is strongly increasing. First of all for every $n$, $\rho_{2n}$ is non-decreasing, and since $\rho_{2n+1} = \rho_{2n}$, neither is $\rho_{2n+1}$. We need to check that for every $x < y \in \omega$ and $s \in \omega$, if $\rho_{s+1}(x) > \rho_s(x)$ then $\rho_{s+1}(y) > s$. By construction of $\rho$, if $\rho_{s+1}(x) > \rho_s(x)$ then $s$ is of the form $2s_1+1$. Let $x_1 = \lfloor x/2 \rfloor$ and $y_1 = \lfloor y/2\rfloor$. In particular $\mu_{s_1+1}(x_1) > \mu_{s_1}(x_1)$, and since $\mu$ is strongly increasing, $\mu_{s_1+1}(y_1) > s_1$ so $\rho_{s+1}(y) = 2\mu_{s_1+1}(y_1) > 2s_1$ so $\rho_{s+1}(y) > s$. Thus $\rho$ is a strongly increasing left-c.e.\ function. Moreover, every function dominating $\rho$ computes $\emptyset'$.

Let $f : [\omega]^{n+1} \to 2^{n+1}$ be defined by $f(D) = \chi(\Lcal_{n+1}(\rho, D))$. By Cholak and Patey~\cite[Corollary 5.5]{Cholak2019Thin} and Lemma~\ref{lem:transitive-to-strongly-transitive}, there is an infinite strongly $\rho$-transitive set $H_{even} \subseteq \{2n : n \in \omega\}$ which does not compute $\emptyset'$, hence does not compute a function dominating $\rho$. We claim that $H = \{x, x+1 : x \in H_{even}\}$ is strongly $\rho$-transitive. Let $w < x < y < z \in H$ be such that $\rho_z(w) > x$ and $\rho_z(x) > y$. Let $w_1, x_1, y_1$ and $z_1$ be the largest even value smaller or equal to $w$, $x$, $y$ and $z$, respectively. 
	Then $\rho_{z_1}(w_1) = \rho_z(w) > x \geq x_1$ and $\rho_{z_1}(x_1) = \rho_z(x) > y \geq y_1$. By strong $\rho$-transitivity of $H_{even}$, $\rho_{z_1}(w_1) > y_1$. In particular, $\rho_z(w) = \rho_{z_1}(w_1) > y_1$, and since $\rho_z(w)$ is even, $\rho_z(w) > y$.
Therefore, $H$ is a $\largep{\chi}^{n+1}_{2^{n+1}}$-solution to $f$ with witness $\rho$.

We claim that $H$ is not an $\rt^{n+1}_{2^{n+1}}(V_{\fs^n})$-solution to $f$. Let $\Gcal \in \Lcal_{n+1}$ be such that $\chi(\Gcal)$ is not the 0-constant function of type $[n+1]^n \to 2$. In particular, there is some $E \in [n+1]^n$ such that $\chi(\Gcal)(E) = 1$.

Since $H_{even}$ does not compute $\emptyset'$, hence does not compute a function dominating $\rho$, by Theorem~\ref{thm:avoid-colors-dominate-left-ce-mu}, there is some $D_{even} = \{x_0 < \dots < x_n\} \in [H_{even}]^{n+1}$ such that $\Gcal = \Lcal_{n+1}(\rho, D_{even})$. 
Let $t$ be the unique element of $\{0, \dots, n+1\} \setminus E$
and let $D = \{x_i : i \in E\} \cup \{x_t+1\}$. In particular, $D \in [H]^{n+1}$. By construction of $\rho$, $\Lcal_{n+1}(\rho, D_{even}) = \Lcal_{n+1}(\rho, D) = \Gcal$, so $f(D) = f(D_{even}) = \chi(\Gcal)$.
Then $D$ and $D_{even} \in [H]^{n+1}$ witness that $H$ is not $f$-functional, and therefore that $H$ is not an $\rt^{n+1}_{2^{n+1}}(V_{\fs^n})$-solution to $f$.
\end{proof}

\subsection{The canonical Ramsey theorem and the rainbow Ramsey theorem}
We conclude this section by providing a more general translation scheme from principles involving $\omega$-colorings of $[\omega]^n$ into promise Ramsey-like statements.

As shown by Theorem~\ref{thm:rtnk-is-universal}, Ramsey's theorem is the maximal true Ramsey-like theorem for finite colorings of $[\omega]^n$.
The following canonical Ramsey theorem can be seen as the maximal true Ramsey-like theorem for $\omega$-colorings from $[\omega]^n$. Given a coloring $f : [\omega]^n \to \omega$, a set $H$ is \emph{$f$-canonical} if there is a set $U \subseteq \{0, \dots, n-1\}$ such that for every $D_0 = \{x_0 < \dots < x_{n-1}\} \in [H]^n$ and $D_1 = \{y_0 < \dots < y_{n-1}\} \in [H]^n$, $f(D_0) = f(D_1)$ if and only if for every $i \in U$, $x_i = y_i$.

\begin{statement}[Canonical Ramsey theorem]
$\crt^n$: Every coloring $f : [\omega]^n \to \omega$ has an infinite $f$-canonical set.
\end{statement}

The canonical Ramsey theorem was first studied by Mileti~\cite{Mileti2004Partition} from a computability-theoretic viewpoint. He proved in particular that $\crt^2$ does not admit cone avoidance. We now explain how to express the canonical Ramsey theorem as a promise Ramsey-like statement.

Given a function $f : [\omega]^n \to \omega$, for every $D = \{x_0 < \dots < x_{2n-1}\} \in [\omega]^{2n}$, let $g(D)$ be the function $g_D : [2n]^n \times [2n]^n \to 2$ defined for every $E_0, E_1 \in [2n]^n$ by
$g_D(E_0, E_1) = 1$ if and only if $f(\{x_i : i \in E_0\}) = f(\{x_i : i \in E_1\})$.

\begin{lemma}
An infinite set $H$ is $g$-homogeneous if and only if it is $f$-canonical.
\end{lemma}
\begin{proof}
$\Rightarrow$: Suppose first that $H$ is $g$-homogeneous, say for color $c : [2n]^n \times [2n]^n \to 2$.

Claim 1: There is a finite set $U \subseteq \{0, \dots, n-1\}$ such that for every $E_0 = \{a_0 < \dots < a_{n-1}\} \in [2n]^n$ and $E_1 = \{b_0 < \dots < b_{n-1}\} \in [2n]^n$, $c(E_0, E_1) = 1$ if and only if for every $i \in U$, $a_i = b_i$.
 Fix $E_0$ and $E_1$.
 By the classical canonical Ramsey theorem, there is an infinite subset $H_1 \subseteq H$ which is $f$-canonical with some witness set $U \subseteq \{0, \dots, n-1\}$. Note that $c(E_0, E_1) = 1$ if and only if for every $F = \{z_0 < \dots < z_{2n-1}\} \in [H_1]^{2n}$, letting $D_0 = \{z_i : i \in E_0\}$ and $D_1 = \{z_i : i \in E_1\}$, $f(D_0) = f(D_1)$. Since $H_1$ is $f$-canonical with witness $U$, then $f(D_0) = f(D_1)$ if and only if for every $i \in U$, the $i$th element of $D_0$ equals the $i$th element of $D_1$, if for every $i \in U$, $a_i = b_i$. This proves Claim 1. From now on, fix the set~$U$.
 
 Claim 2: $H$ is $f$-canonical with witness $U$. Fix some $D_0, D_1 \in [H]^n$, and let $F = \{z_0 < \dots < z_{2n-1}\} \in [H]^{2n}$ be some set such that $D_0 \cup D_1 \subseteq F$. Let $E_0 = \{i < 2n : z_i \in D_0\}$ and $E_1 = \{i < 2n : z_i \in D_1\}$. By definition of $g$, $f(D_0) = f(D_1)$ if and only if $c(E_0, E_1) = 1$. By Claim 1, $c(E_0, E_1) = 1$ if and only if for every $i \in U$, the $i$th element of $E_0$ equals the $i$th element of $E_1$. Therefore, $f(D_0) = f(D_1)$ if and only if for every $i \in U$, the $i$th element of $D_0$ equals the $i$th element of $D_1$. This proves Claim 2.
 
$\Leftarrow$: Suppose now that $H$ is $f$-canonical, with witness $U \subseteq \{0, \dots, n-1\}$. We claim that $H$ is $g$-homogeneous. Fix some $D = \{z_0 < \dots < z_{2n-1}\} \in [H]^{2n}$ and let $c = f(D)$, with $c : [2n]^n \times [2n]^n \to 2$. We claim that $c$ is fully specified by $U$. Fix some $E_0, E_1 \in [2n]^n$ and let $D_0 = \{z_i : i \in E_0\}$ and $D_1 = \{z_i : i \in E_1\}$. Then $c(E_0, E_1) = 1$ if and only if $f(D_0) = f(D_1)$, and by $f$-canonicity of $H$, this holds if and only if for every $i \in U$, the $i$th element of $D_0$ equals the $i$th element of $D_1$, which again holds if and only if the $i$th element of $E_0$ equals the $i$th element of $E_1$.
This property depends only on $E_0$, $E_1$ and $U$. Therefore $c$ is unique.
\end{proof}

We say that a set $H$ is \emph{$g$-comparing} if $g$ behaves as the coding of some function $f$, that is, for every $F_{1,2} = \{x_0 < \dots < x_{2n-1}\} \in [H]^{2n}$, $F_{1,3} = \{y_0 < \dots < y_{2n-1}\} \in [H]^{2n}$ and $F_{2,3} = \{z_0 < \dots < z_{2n-1}\} \in [H]^{2n}$ and $E^1_{1,2}, E^1_{1,3}, E^2_{1,2}, E^2_{2,3}, E^3_{1,3}, E^3_{2,3}$ such that $\{x_i : i \in E^1_{1,2} \} = \{ y_i : i \in E^1_{1,3} \}$, $\{x_i : i \in E^2_{1,2} \} = \{ z_i : i \in E^2_{2,3} \}$ and $\{y_i : i \in E^3_{1,3}\} = \{z_i : i \in E^3_{2,3}\}$, if $g(F_{1,2})(E^1_{1,2},E^2_{1,2}) = g(F_{2,3})(E^2_{2,3}, E^3_{2,3}) = 1$ then $g(F_{1,3})(E^1_{1,3}, E^3_{1,3}) = 1$. Moreover $g(F_{1,2})(E^1_{1,2},E^2_{1,2}) = 1$ and if $g(F_{1,2})(E^1_{1,2},E^2_{1,2}) = 1$ then $g(F_{1,2})(E^2_{1,2},E^1_{1,2}) = 1$.
Let $V_{\crt^n}$ be the set of $\rt^{2n}_\ell$-patterns (where $\ell$ is the number of functions of type $[2n]^n \times [2n]^n \to 2$) forbidding the sets which are non $g$-comparing. Then $\crt^n$ can be seen as the promise Ramsey-like statement $\rt^{2n}_\ell(V_{\crt^n}, W_{\rt^{2n}_\ell})$.

A function $f : [\omega]^n \to \omega$ is \emph{$k$-bounded} if for every $c \in \omega$, $|f^{-1}(c)| \leq k$, that is, each color appears at most $k$ times. A set $H \subseteq \omega$ is an \emph{$f$-rainbow} if $f$ is injective over $[H]^n$. 

\begin{statement}[Rainbow Ramsey theorem]
$\rrt^n_k$: Every $k$-bounded coloring $f : [\omega]^n \to \omega$ has an infinite $f$-rainbow.
\end{statement}

Cisma and Mileti~\cite{Csima2009strength} first studied the rainbow Ramsey theorem in the context of reverse mathematics. Wang~\cite[Theorem 4.2]{Wang2014Some} proved that $\rrt^n_2$ follows directly from $\fs^n$, thus that $\rrt^n_2$ admits strong cone avoidance for every $n \geq 1$. Later, the author~\cite[Theorem 4.6]{Patey2015Somewhere} proved that $\fs^n$ follows from $\rrt^{2n+1}_2$.
Given a $k$-bounded function $f : [\omega]^n \to \omega$, we can define $g : [\omega]^{2n} \to [2n]^n \times [2n]^n \to 2$ as for the canonical Ramsey theorem. An infinite set $H$ is $g$-homogeneous if and only if it is an $f$-rainbow. Then letting $V_{\rrt^n_k}$ be the set of $\rt^{2n}_\ell$-patterns which force the sets to be $g$-comparing and to code a $k$-bounded function, the rainbow Ramsey theorem can be seen as the promise Ramsey-like statement $\rt^{2n}_\ell(V_{\rrt^n_k}, W_{\rt^{2n}_\ell})$.

\section{Open questions}\label{sect:open-questions}

This article provides an extensive analysis of cone avoidance for Ramsey-like and promise Ramsey-like theorems. Other weakness notions have been proven to be very useful in reverse mathematics. It would be interesting to extend this analysis to these notions. We detail some of the remaining questions.

\subsection{PA avoidance} Among the five main subsystems of second-order arithmetics studied in reverse mathematics, weak K\"onig's lemma ($\wkl$) captures compactness arguments. Weak K\"onig's lemma asserts that every infinite binary tree admits an infinite path. The question whether $\rt^2_2$ implies $\wkl$ was a long-standing open question, until Liu~\cite{Liu2012RT22} answered it negatively using the notion of PA avoidance. A Turing degree $\mathbf{d}$ is \emph{PA} relative to $X$ if every infinite $X$-computable binary tree has an infinite path bounded by $\mathbf{d}$. 

\begin{definition}[PA avoidance]
A problem $\Psf$ admits \emph{PA avoidance} if for every set $Z$ of non-PA degree and every $Z$-computable $\Psf$-instance $X$, there is a $\Psf$-solution $Y$ to $X$ such that $Z \oplus Y$ is of non-PA degree.
\end{definition}

 The notion of strong PA avoidance is defined accordingly. Liu~\cite{Liu2012RT22} proved that $\rt^1_2$ admits strong PA avoidance and deduced that $\rt^2_2$ admits PA avoidance.

\begin{question}
What Ramsey-like statements admit PA and strong PA avoidance, respectively?
\end{question}

The cone avoidance analysis for Ramsey-like statements strongly relies on finding the exact bounds for which the thin set theorems admits cone avoidance. The author~\cite{PateyCombinatorial} proved that for every $n \geq 1$, there is some $\ell \in \omega$ such that $\rt^n_{<\infty, \ell}$ admits strong PA avoidance. 

\subsection{Preservation of hyperimmunities}
A very important and successful computability-theoretic notion to separate statements in reverse mathematics is simultaneous preservation of hyperimmunities. A function $f : \omega \to \omega$ is \emph{$X$-hyperimmune} if it is not dominated by any $X$-computable function.

\begin{definition}[Preservation of hyperimmunities]
A problem $\Psf$ admits \emph{preservation of $k$ hyperimmunities} if for every set $Z$, every $k$-tuple of $Z$-hyperimmune functions $f_0, \dots, f_{k-1}$ and every $Z$-computable $\Psf$-instance $X$, there is a $\Psf$-solution $Y$ such that all the functions $f_0, \dots, f_{k-1}$ are $Z \oplus Y$-hyperimmune.
\end{definition}

Again, the notion of strong preservation of $k$ hyperimmunities is defined accordingly. The analysis of Section~\ref{sect:ramsey-like} actually shows that whenever $\rt^n_k(W) \not \leq_{id} \largeo^n_k$, then $\rt^n_k(W)$ does not admit strong preservation of 1 hyperimmunity. Actually, the proof of Cholak and Patey~\cite{Cholak2019Thin} that $\rt^n_{<\infty, C_n}$ admits strong cone avoidance can be adapted to prove that $\rt^n_{<\infty, C_n}$ admits strong preservation of 1 hyperimmunity. By a similar analysis, we can prove that $\largeo^n_k$ admits strong preservation of 1 hyperimmunity, and thus that the Ramsey-like statements which admit strong preservation of 1 hyperimmunity and strong cone avoidance coincide. The situation becomes different when preserving 2 hyperimmunities. Indeed, the author~\cite[Lemma 25 and Lemma 27]{Patey2015Iterative} proved that $\rt^1_2$ does not admit strong preservation of 2 hyperimmunities, while Dzhafarov and Jockusch~\cite{Dzhafarov2009Ramseys} proved that $\rt^1_2$ can strongly avoid multiple cones simultaneously.

\begin{question}
What Ramsey-like statements admit preservation and strong preservation of $k$ hyperimmunities, respectively?
\end{question}

The author~\cite[Theorem 8.4.1]{Patey2016reverse} proved that for every $k \in \omega$ and $n \geq 1$, there is some $\ell \in \omega$ such that $\rt^n_{<\infty, \ell}$ admits preservation of $k$ hyperimmunities.

\subsection{Jump cone avoidance}
The question of the relation between stable Ramsey's theorem for pairs and cohesiveness~\cite{Patey2016Open} motivated the study of jump computation and yielded the notion of jump cone avoidance which is similar to the notion of cone avoidance, but for jump computation.

\begin{definition}[Jump cone avoidance]
A problem $\Psf$ admits \emph{jump cone avoidance} if for every set $Z$, every non-$\Delta^0_2(Z)$ set $C$, and every $Z$-computable $\Psf$-instance $X$, there is a $\Psf$-solution $Y$ such that $C$ is not $\Delta^0_2(Z \oplus Y)$.
\end{definition}

Once again, the notion of strong jump cone avoidance is defined accordingly by dropping the effectiveness restraint on the $\Psf$-instance.
Recently, Monin and Patey~\cite[Theorem 4.1]{Monin2018Pigeons} proved that $\rt^1_2$ admits strong jump cone avoidance.
It is however currently unknown whether for every $n \geq 1$, there is some $\ell \in \omega$ such that $\rt^n_{<\infty, \ell}$ admits strong jump cone avoidance. 

\begin{question}
What Ramsey-like statements admit jump cone and strong jump cone avoidance, respectively?
\end{question}

\bibliographystyle{plainnat}
\bibliography{bibliography}

\begin{thebibliography}{27}
\providecommand{\natexlab}[1]{#1}
\providecommand{\url}[1]{\texttt{#1}}
\expandafter\ifx\csname urlstyle\endcsname\relax
  \providecommand{\doi}[1]{doi: #1}\else
  \providecommand{\doi}{doi: \begingroup \urlstyle{rm}\Url}\fi

\bibitem[Bovykin and Weiermann(2005)]{Bovykin2005strength}
Andrey Bovykin and Andreas Weiermann.
\newblock The strength of infinitary {R}amseyan principles can be accessed by
  their densities.
\newblock \emph{Annals of Pure and Applied Logic}, page~4, 2005.
\newblock To appear.

\bibitem[Cholak and Patey(2019)]{Cholak2019Thin}
Peter~A. Cholak and Ludovic Patey.
\newblock Thin set theorems and cone avoidance.
\newblock To appear., 2019.

\bibitem[Cholak et~al.(2001{\natexlab{a}})Cholak, Giusto, Hirst, and
  Jockusch~Jr]{Cholak2001Free}
Peter~A. Cholak, Mariagnese Giusto, Jeffry~L. Hirst, and Carl~G. Jockusch~Jr.
\newblock Free sets and reverse mathematics.
\newblock \emph{Reverse mathematics}, 21:\penalty0 104--119,
  2001{\natexlab{a}}.

\bibitem[Cholak et~al.(2001{\natexlab{b}})Cholak, Jockusch, and
  Slaman]{Cholak2001strength}
Peter~A. Cholak, Carl~G. Jockusch, and Theodore~A. Slaman.
\newblock {On the strength of Ramsey's theorem for pairs}.
\newblock \emph{Journal of Symbolic Logic}, 66\penalty0 (01):\penalty0 1--55,
  2001{\natexlab{b}}.

\bibitem[Csima and Mileti(2009)]{Csima2009strength}
Barbara~F. Csima and Joseph~R. Mileti.
\newblock {The strength of the rainbow {R}amsey theorem}.
\newblock \emph{Journal of Symbolic Logic}, 74\penalty0 (04):\penalty0
  1310--1324, 2009.

\bibitem[Csima et~al.(2004)Csima, Hirschfeldt, Knight, and
  Soare]{Csima2004Bounding}
Barbara~F. Csima, Denis~R. Hirschfeldt, Julia~F. Knight, and Robert~I. Soare.
\newblock Bounding prime models.
\newblock \emph{Journal of Symbolic Logic}, pages 1117--1142, 2004.

\bibitem[Dorais et~al.(2016)Dorais, Dzhafarov, Hirst, Mileti, and
  Shafer]{Dorais2016uniform}
Fran{\c{c}}ois~G. Dorais, Damir~D. Dzhafarov, Jeffry~L. Hirst, Joseph~R.
  Mileti, and Paul Shafer.
\newblock On uniform relationships between combinatorial problems.
\newblock \emph{Trans. Amer. Math. Soc.}, 368\penalty0 (2):\penalty0
  1321--1359, 2016.
\newblock ISSN 0002-9947.
\newblock \doi{10.1090/tran/6465}.
\newblock URL \url{http://dx.doi.org/10.1090/tran/6465}.

\bibitem[Dzhafarov and Jockusch(2009)]{Dzhafarov2009Ramseys}
Damir~D. Dzhafarov and Carl~G. Jockusch.
\newblock Ramsey's theorem and cone avoidance.
\newblock \emph{Journal of Symbolic Logic}, 74\penalty0 (2):\penalty0 557--578,
  2009.

\bibitem[Friedman()]{FriedmanFom53free}
Harvey~M. Friedman.
\newblock Fom:53:free sets and reverse math and fom:54:recursion theory and
  dynamics.
\newblock URL \url{http://www.math.psu.edu/simpson/fom/}.
\newblock Available at \url{https://www.cs.nyu.edu/pipermail/fom/}.

\bibitem[Groszek and Slaman(2007)]{Groszek2007Moduli}
Marcia~J Groszek and Theodore~A Slaman.
\newblock Moduli of computation (talk).
\newblock \emph{Buenos Aires, Argentina}, 2007.

\bibitem[Hirschfeldt and Jockusch(2016)]{Hirschfeldt2016notions}
Denis~R. Hirschfeldt and Carl~G. Jockusch.
\newblock On notions of computability-theoretic reduction between {$\Pi\sb 2\sp
  1$} principles.
\newblock \emph{J. Math. Log.}, 16\penalty0 (1):\penalty0 1650002, 59, 2016.
\newblock ISSN 0219-0613.
\newblock \doi{10.1142/S0219061316500021}.
\newblock URL \url{http://dx.doi.org/10.1142/S0219061316500021}.

\bibitem[Hirschfeldt and Shore(2007)]{Hirschfeldt2007Combinatorial}
Denis~R. Hirschfeldt and Richard~A. Shore.
\newblock {Combinatorial principles weaker than {R}amsey's theorem for pairs}.
\newblock \emph{Journal of Symbolic Logic}, 72\penalty0 (1):\penalty0 171--206,
  2007.

\bibitem[Jockusch(1972)]{Jockusch1972Ramseys}
Carl~G. Jockusch.
\newblock Ramsey's theorem and recursion theory.
\newblock \emph{Journal of Symbolic Logic}, 37\penalty0 (2):\penalty0 268--280,
  1972.

\bibitem[Lerman(1983)]{Lerman1983Degrees}
Manuel Lerman.
\newblock Degrees of unsolvability: Local and global theory. perspectives in
  mathematical logic.
\newblock 1983.

\bibitem[Liu(2012)]{Liu2012RT22}
Lu~Liu.
\newblock {RT$^2_2$ does not imply WKL$_0$}.
\newblock \emph{Journal of Symbolic Logic}, 77\penalty0 (2):\penalty0 609--620,
  2012.

\bibitem[Mileti(2004)]{Mileti2004Partition}
Joseph~Roy Mileti.
\newblock \emph{Partition theorems and computability theory}.
\newblock ProQuest LLC, Ann Arbor, MI, 2004.
\newblock ISBN 978-0496-13959-0.
\newblock URL
  \url{http://gateway.proquest.com/openurl?url_ver=Z39.88-2004&rft_val_fmt=info:ofi/fmt:kev:mtx:dissertation&res_dat=xri:pqdiss&rft_dat=xri:pqdiss:3153383}.
\newblock Thesis (Ph.D.)--University of Illinois at Urbana-Champaign.

\bibitem[Monin and Patey(2018)]{Monin2018Pigeons}
Benoit Monin and Ludovic Patey.
\newblock Pigeons do not jump high.
\newblock To appear. Available at \url{https://arxiv.org/abs/1803.09771}, 2018.

\bibitem[Patey(2015{\natexlab{a}})]{Patey2015Iterative}
Ludovic Patey.
\newblock Iterative forcing and hyperimmunity in reverse mathematics.
\newblock In Arnold Beckmann, Victor Mitrana, and Mariya Soskova, editors,
  \emph{CiE. Evolving Computability}, volume 9136 of \emph{Lecture Notes in
  Computer Science}, pages 291--301. Springer International Publishing,
  2015{\natexlab{a}}.
\newblock ISBN 978-3-319-20027-9.
\newblock \doi{10.1007/978-3-319-20028-6_30}.
\newblock URL \url{http://dx.doi.org/10.1007/978-3-319-20028-6_30}.

\bibitem[Patey(2015{\natexlab{b}})]{Patey2015Somewhere}
Ludovic Patey.
\newblock {Somewhere over the rainbow {R}amsey theorem for pairs}.
\newblock Submitted. Available at \url{http://arxiv.org/abs/1501.07424},
  2015{\natexlab{b}}.

\bibitem[Patey(2015{\natexlab{c}})]{PateyCombinatorial}
Ludovic Patey.
\newblock Combinatorial weaknesses of {R}amseyan principles.
\newblock In preparation. Available at
  \url{http://ludovicpatey.com/media/research/combinatorial-weaknesses-draft.pdf},
  2015{\natexlab{c}}.

\bibitem[Patey(2016{\natexlab{a}})]{Patey2016Open}
Ludovic Patey.
\newblock Open questions about {R}amsey-type statements in reverse mathematics.
\newblock \emph{Bull. Symb. Log.}, 22\penalty0 (2):\penalty0 151--169,
  2016{\natexlab{a}}.
\newblock ISSN 1079-8986.
\newblock \doi{10.1017/bsl.2015.40}.
\newblock URL \url{https://doi.org/10.1017/bsl.2015.40}.

\bibitem[Patey(2016{\natexlab{b}})]{Patey2016reverse}
Ludovic Patey.
\newblock \emph{The reverse mathematics of Ramsey-type theorems}.
\newblock PhD thesis, UniversitÃ© Paris Diderot, 2016{\natexlab{b}}.

\bibitem[Patey(2017)]{Patey2017Iterative}
Ludovic Patey.
\newblock Iterative forcing and hyperimmunity in reverse mathematics.
\newblock \emph{Computability}, 6\penalty0 (3):\penalty0 209--221, 2017.
\newblock ISSN 2211-3568.
\newblock \doi{10.3233/COM-160062}.
\newblock URL \url{https://doi.org/10.3233/COM-160062}.

\bibitem[Seetapun and Slaman(1995)]{Seetapun1995strength}
David Seetapun and Theodore~A. Slaman.
\newblock {On the strength of {R}amsey's theorem}.
\newblock \emph{Notre Dame Journal of Formal Logic}, 36\penalty0 (4):\penalty0
  570--582, 1995.

\bibitem[Simpson(2009)]{Simpson2009Subsystems}
Stephen~G. Simpson.
\newblock \emph{{Subsystems of Second Order Arithmetic}}.
\newblock Cambridge University Press, 2009.

\bibitem[Solovay(1978)]{Solovay1978Hyperarithmetically}
Robert~M. Solovay.
\newblock Hyperarithmetically encodable sets.
\newblock \emph{Trans. Amer. Math. Soc.}, 239:\penalty0 99--122, 1978.
\newblock ISSN 0002-9947.

\bibitem[Wang(2014)]{Wang2014Some}
Wei Wang.
\newblock Some logically weak {R}amseyan theorems.
\newblock \emph{Advances in Mathematics}, 261:\penalty0 1--25, 2014.

\end{thebibliography}

\end{document}